\documentclass[12pt,a4paper]{article}
\usepackage{amsmath}
\usepackage{amsfonts,dsfont}
\usepackage{amssymb}
\usepackage{setspace}
\usepackage{amsthm}
\usepackage{rotating,enumerate}
\usepackage[a4paper, left=2.5cm, right=2.5cm, top=2.8cm, bottom=2.8cm]{geometry}
\usepackage{float}
\usepackage{amsmath}
\usepackage{graphicx}
\usepackage{subfigure,bm}
\usepackage{natbib}
\usepackage{marginnote} %Package helping with margin notes within equations
\newcommand\mn[1]{
}

\newcommand\labmarg[1]{\label{#1}
\mn{\tiny $#1$}
}
\usepackage{color}

\usepackage{ulem}

\newcommand{\bol}[1]{\mbox{\boldmath$#1$}}
\newcommand{\bmu}{\bol{\mu}}
\newcommand{\btheta}{\bol{\theta}}

\newcommand{\bSigma}{\mathbf{\Sigma}}

\newcommand{\bL}{\mathbf{L}}

\newcommand{\bx}{\mathbf{x}}
\newcommand{\bl}{\mathbf{l}}

\newcommand{\bS}{\mathbf{S}}

\newtheorem{theorem}{Theorem}

\newtheorem{lemma}{Lemma}
\newtheorem{corollary}{Corollary}

\begin{document}

\begin{center}
\vspace*{2cm} \noindent {\bf \large Discriminant analysis in small and large dimensions}\\
\vspace{1cm} \noindent {\sc Taras Bodnar$^{a,}$\footnote{Corresponding Author: Taras Bodnar. E-Mail: taras.bodnar@math.su.se. Tel: +46 8 164562. Fax: +46 8 612 6717. This research was partly supported by the Swedish International Development Cooperation Agency (SIDA) through the UR-Sweden Programme for Research, Higher Education and Institutional Advancement. Stepan Mazur acknowledges financial support from the project ''Ambit fields: probabilistic properties and statistical inference'' funded by Villum Fonden}, Stepan Mazur$^{b}$, Edward Ngailo$^{a}$ and Nestor Parolya$^{c}$
}\\
\vspace{1cm}
{\it \footnotesize  $^a$
 Department of Mathematics, Stockholm University, Roslagsv\"{a}gen 101, SE-10691 Stockholm, Sweden}\\
{\it \footnotesize  $^b$ Unit of Statistics, School of Business, \"{O}rebro University, Fakultetsgatan 1, SE-70182 \"{O}rebro, Sweden
} \\
{\it \footnotesize  $^c$ Institute of Statistics, Leibniz University Hannover, K\"{o}nigsworther Platz 1, D-30167 Hannover, Germany
} \\

\end{center}

\begin{abstract}
We study the distributional properties of the linear discriminant function under the assumption of normality by comparing two groups with the same covariance matrix but different mean vectors. A stochastic representation for the discriminant function coefficients is derived which is then used to obtain their asymptotic distribution under the high-dimensional asymptotic regime. We investigate the performance of the classification analysis based on the discriminant function in both small and large dimensions. A stochastic representation is established which allows to compute the error rate in an efficient way. We further compare the calculated error rate with the optimal one obtained under the assumption that the covariance matrix and the two mean vectors are known. Finally, we present an analytical expression of the error rate calculated in the high-dimensional asymptotic regime. The finite-sample properties of the derived theoretical results are assessed via an extensive Monte Carlo study.
\end{abstract}

\noindent ASM Classification: 62H10, 62E15, 62E20, 60F05, 60B20

\newpage
\section{Introduction}
In the modern world of science and technology, high-dimensional data are present in various fields such as finance, environment science and social sciences. In the sense of many complex multivariate dependencies observed in data, formulating correct models and developing inferential procedures are the major challenges. The traditional multivariate analysis considers fixed or small sample dimensions, while sample sizes approaching to infinity. However, its methods cannot longer be used in the high-dimensional setting where the dimension is not treated as fixed but it is allowed to be comparable to the sample size.

The covariance matrix is one of the mostly used way to capture the dependence between variables. Although its application is restricted only to linear dependence and more sophisticated methods, like copula, should be applied in the general case, modeling dynamics in the covariance matrix is still a very popular subject in both statistics and econometrics. Recently, a number of papers have been published which deal with estimating the covariance matrix (see, e.g., \citet{ledoit2003improved}, \citet{cai2011adaptive}, \citet{cai2011constrained}, \citet{agarwal2012noisy}, \citet{fan2008high}, \citet{fan2013large}, \cite{BodnarGuptaParolya2014,BGP2016}) and testing its structure (see, e.g., \citet{johnstone2001distribution}, \citet{bai2009corrections}, \citet{chen2010tests}, \citet{cai2011limiting}, \citet{jiang2013central}, \citet{GuptaBodnar2014}) in large dimension.

In many applications, the covariance matrix is accompanied by the mean vector. For example, the product of the inverse sample covariance matrix and the difference of the sample mean vectors is present in the discriminant function where a linear combination of variables (discriminant function coefficients) is determined such that the standardized distance between the groups of observations is maximized. A second example arises in portfolio theory, where the vector of optimal portfolio weights is proportional to the products of inverse sample covariance matrix and the sample mean vector (see \cite{BOK2011}).

The discriminant analysis is a multivariate technique concerned with separating distinct sets of objects (or observations) (\cite{JRO2007}). Its two main tasks are to distinguish distinct sets of observations and to allocate new observations to previously defined groups (\cite{RC2012}). The main methods of the discriminant analysis are the linear discriminant function and the quadratic discriminant function. The linear discriminant function is a generalization of Fisher linear discriminant analysis, a method used in statistics, pattern recognition and machine learning to find a linear combination of features that characterizes or separates two or more groups of objects in the best way. The application of the linear discriminant function is restricted to the assumption of the equal covariance matrix in the groups to be separated. Although the quadratic discriminant function can be used when the latter assumption is violated, its application is more computational exhaustive, needs to estimate the covariance matrices of each group, and requires more observations than in the case of linear discriminant function (\cite{NF2013}). Moreover, the decision boundary is easy to understand and to visualize in high-dimensional settings, if the  linear discriminant function is used.

The discriminant analysis is a well established topic in multivariate statistics. Many asymptotic results are available when the sample sizes of groups to be separated are assumed to be large, while the number of variables is fixed and significantly smaller than the sample size (see, e.g., \cite{M1982}, \cite{RC2012}). However, these results cannot automatically be transferred when the number of variables is comparable to the sample size which is known in the statistical literature as the high-dimensional asymptotic regime. It is remarkable that in this case the results obtained under the standard asymptotic regime can deviate significantly from those obtained under the high-dimensional asymptotics (see, e.g., \citet{BaiSilverstein2010}). \cite{fujikoshi1997asymptotic} provided an asymptotic approximation of the linear discriminant function in high dimension by considering the case of equal sample sizes and compared the results with the classical asymptotic approximation by \cite{wyman1990comparison}. For the samples of non-equal sizes, they pointed out that the high-dimensional approximation is extremely accurate. However, \cite{tamatani2015asymptotic} showed that the Fisher linear discriminant function performs poorly due to diverging spectra in the case of large-dimensional data and small sample sizes. \cite{bickel2004some}, \cite{srivastava2007comparison} investigated the asymptotic properties of the linear discriminant function in high dimension, while modifications of the linear discriminant function can be found in \cite{cai2011direct}, \cite{shao2011sparse}. The asymptotic results for the discriminant function coefficients in matrix-variate skew models can be found in \cite{BMP2016}.

We contribute to the statistical literature by deriving a stochastic representation of the discriminant function coefficient and the classification rule based on the linear discriminant function. These results provide us an efficient way of simulating these random quantities  and they are also used in the derivation of their high-dimensional asymptotic distributions, using which the error rate of the classification rule based on the linear discriminant function can be easily assessed and the problem of the increasing dimensionality can be visualized in a simple way.

The rest of the paper is organized as follows. The finite-sample properties of the discriminant function are presented in Section 2.1, where, in particular we derive a stochastic representation for the discriminant function coefficients. In Section 2.2, an exact one-sided test for the comparison of the population discriminant function coefficients is suggested, while a stochastic representation for the classification rule is obtained in Section 2.3. The finite-sample results are then use to derive the asymptotic distributions of the discriminant function coefficients and of the classification rule in Section 3, while finite sample performance of the asymptotic distribution is analysed in Section 3.2.

%-------------------------------------------------------------------
\section{Finite-sample properties of the discriminant function}
\labmarg{s2}
%-------------------------------------------------------------------
Let $\mathbf x _1^{(1)} ,  \dots  , \mathbf x _{n_{1}} ^ {(1)}$ and $\mathbf x_{1} ^ {(2)} , \dots , \mathbf x _ {n_{2}} ^ {(2)}$ be two independent samples from the multivariate normal distributions which consist of independent and identically distributed random vectors with $\mathbf x _i^{(1)}\sim \mathcal{N}_p(\bmu_1,\bSigma)$ for $i=1,...,n_1$ and $\mathbf x _j^{(2)}\sim \mathcal{N}_p(\bmu_2,\bSigma)$ for $j=1,...,n_2$ where $\bSigma$ is positive definite. Throughout the paper, $\mathbf 1_ n $ denotes the $n$-dimensional vector of ones, $\mathbf I _n$ is the $n\times n$ identity matrix, and the symbol $\otimes$ stands for the Kronecker product.

Let $\mathbf X ^{(1)} = \left(   \mathbf x _ 1 ^{(1)}  , \dots, \mathbf x _{n_{1}} ^  {(1)}  \right)$
and $\mathbf X ^{(2)} =\left(    \mathbf x _ 1 ^{(2)} ,\dots ,\mathbf x _{n_{2}} ^  {(2)}    \right)$ be observation matrices.
Then the sample estimators for the mean vectors and the covariance matrices constructed from each sample are given by
\begin{eqnarray*}
\bar{ \mathbf x } ^  {(j)}
&=&
\frac{1}{n_{j}}\sum\limits_{i=1}^{n_{j}}\mathbf x_i^{(j)}
=
\frac{1}{n_j}
\mathbf X ^{(j)}   \mathbf{1}_{n_j}
 \\
\mathbf S ^{(j)}
&=&
\frac{1}{n_{j}-1}
\sum_{i=1}^{n_j}
\left( \mathbf x _i^{(j)} - \bar{\mathbf x }^{(j)}\right)
\left( \mathbf x_i^{(j)} - \bar{\mathbf x }^{(j)}\right)^T\,.
%= \frac{1}{n_{j}-1} \mathbf X ^{(j)} \mathbf V ^{(j)}  \mathbf X ^{(j)T}
\end{eqnarray*}
The pooled estimator for the covariance matrix, i.e., an estimator for $\bSigma$ obtained from two samples, is then given by
\begin{eqnarray}
\mathbf S_{pl}= \frac{1}{n_{1}+n_{2}-2} \left[ (n_{1}-1)\mathbf S^{(1)}+(n_{2}-1)\mathbf S^{(2)}\right]
\end{eqnarray}

The following lemma (see, e.g., \cite[Section 5.4.2]{RC2012}) presents the joint distribution of $\bar{\mathbf x}^{(1)}$, $\bar{\mathbf x}^{(2)}$ and  $\mathbf S_{pl}$.

%--------------------------------------------------------------
\begin{lemma}
\label{l1}
Let $\mathbf X_1 \sim \mathcal{N}_{p, n_1}  \left(\bmu_1 \mathbf 1_{n_1}^T , \bSigma \otimes \mathbf I_{n_1} \right)$ and $\mathbf X_2 \sim \mathcal{N}_{p, n_2}  \left(\bmu_2 \mathbf 1_{n_2}^T , \bSigma \otimes \mathbf I_{n_2} \right)$ for $ p < n_1 + n_2 - 2$. Assume that $\mathbf X_1$ and $\mathbf X_2 $ are independent. Then
\begin{enumerate}
\item[(a)]
$\bar{\mathbf x }^{(1)}
\sim
\mathcal{N}_p
\left(
\bmu_1,
\frac{1}{n_1}
\bSigma
\right)$,
\item[(b)]
$\bar{\mathbf x }^{(2)}
\sim
\mathcal{N}_p
\left(
\bmu_2,
\frac{1}{n_2}
\bSigma
\right)$,

\item[(c)]
$ (n_{1}+n_{2}-2)\mathbf S_{pl}
\sim
\mathcal W_p(n_{1}+n_{2}-2, \bSigma)$,
\end{enumerate}
Moreover, $\bar{\mathbf x }^{(1)}$, $\bar{\mathbf x }^{(2)}$ and $\mathbf S_{pl}$ are mutually independently distributed.
\end{lemma}

The results of Lemma \ref{l1}, in particular, implies that
\begin{eqnarray}
\bar{\mathbf x } ^{(1)} - \bar{\mathbf x } ^{(2)}\sim \mathcal N_p\left( \bmu_1 -\bmu_2,\left(\frac{1}{n_1} + \frac{1}{n_2}\right) \bSigma\right)
\end{eqnarray}
which is independent of $\mathbf S_{pl}$.

%%%%%%%%%%%%%%%%%%%%%%%%%%%%%%%%%%%%%%
%   Section DF
%%%%%%%%%%%%%%%%%%%%%%%%%%%%%%%%%%%%%%

%\section{Linear discriminant function and its stochastic representations}\labmarg{s3}

%Formally we look for a linear combination of the variables, such that the standardized distance between the groups of observations is maximized.
%Let $\mathbf a=(a_{1}, a_{2},\ldots,a_{p}) \in \mathbb R^{p}$ be a vector of constants
%and $z_ i ^ {(1)} = \mathbf a ^ T \mathbf x _ i ^{(1)} \ ( i = \overline{ 1, n_1} )$
% and similarly  $z_ i^{(2)} = \mathbf a ^ T \mathbf x _ i ^{(2)}  \ ( i = \overline{ 1, n_2 } )$.
% The corresponding sample means are denoted by  $\overline z ^j = \mathbf a ^ T \overline{\mathbf x } ^{(j)}
% \ (j=\{1,2\})$, while the corresponding sample variance is denoted by $\sigma_z^2 = \mathbf a^T \mathbf S_{pl} \mathbf a$  .
%Since we maximize the standardized difference
%\begin{eqnarray}
%\frac{ \left(\overline z ^{(1)} - \overline z ^{(2)}\right)^2 }
%{\sigma_z^2}
%=\frac{\left[\mathbf a^{T}\left(\bar {\mathbf x}^{(1)} -\bar {\mathbf x}^{(2)}\right)\right]^{2}}{\mathbf a^{T}\mathbf S_{pl}\mathbf a}
%=\frac{\mathbf a^T \left(\bar{\mathbf x}^{(1)}-\bar{\mathbf x}^{(2)}\right) \left(\bar{\mathbf x}^{(1)}-\bar{\mathbf x}^{(2)}\right)^{T}\mathbf a}{\mathbf a^{T}\mathbf S_{pl}\mathbf a},
%\end{eqnarray}
% then the maximum is achieved at

\subsection{Stochastic representation for the discriminant function coefficients}
The discriminant function coefficients are given by the following vector
 \begin{eqnarray}
 \hat{\mathbf a} = \mathbf S ^{-1}_{pl}\left( \bar{\mathbf x }^{(1)} - \bar{\mathbf x }^{(2)}\right)
 \end{eqnarray}
which is the sample estimator of the population discriminant function coefficient vector expressed as
\begin{eqnarray*}
 \mathbf a = \bSigma ^{-1}\left( \bmu_1-\bmu_2\right)
\end{eqnarray*}

We consider a more general problem by deriving the distribution of linear combinations of the discriminant function coefficients. This result possesses several practical application: (i) it allows a direct comparison of the population coefficients in the discriminant function by deriving a corresponding statistical test; (ii) it can be used in the classification problem where providing a new observation vector one has to decide to which of two groups the observation vector has to be ordered.

Let $\bL$ be a $k \times p$ matrix of constants such that $rank (\mathbf L) = k < p$. We are then interested in
\begin{eqnarray}
\hat \btheta=\mathbf L \hat{\mathbf a}
=\mathbf L  \mathbf S ^{-1}_{pl}\left( \bar{\mathbf x }^{(1)} - \bar{\mathbf x }^{(2)}\right).
\end{eqnarray}
Choosing different matrices $\bL$ we are able to provide different inferences about the linear combinations of the discriminant function coefficients. For instance, if $k=1$ and $\bL$ is the vector with all elements zero except the one on the $j$th position which is one, then we get the distribution of the $j$th coefficient in the discriminant function. If we choose $k=1$ and $\bL=(1,-1,0,\ldots,0)^T$, then we analyse the difference between the first two coefficients in the discriminant function. The corresponding result can be further used to test if the population counterparts to these coefficients are zero or not. For $k>1$ several linear combinations of the discriminant function coefficients are considered simultaneously.

In the next theorem we derive a stochastic representation for $\hat \btheta$. The stochastic representation is a very important tool in analysing the distributional properties of random quantities. It is widely spread in the computation statistics (e.g., \cite{givens2012computational}), in the theory of elliptical distributions (see, \cite{Gupta2013}) as well as in Bayesian statistics (cf., \cite{BodnarMazurOkhrin2016}). Later on, we use the symbol ${\buildrel d \over =}$ to denote the equality in distribution.

%---------------------------------
%        Theorem 1
%---------------------------------

\begin{theorem} \labmarg{th1}
Let $\mathbf L $ be an arbitrary $k\times p$ matrix of constants such that $rank (\mathbf L) = k < p$. Then, under the assumption of Lemma \ref{l1} the stochastic representation of $\hat \btheta = \mathbf L \hat{ \mathbf a}$ is given by
\begin{eqnarray}
\hat \btheta
&
{\buildrel d \over =}
&
(n_1+n_2-2)
\xi^{-1}
\left(
\mathbf L \bSigma^{-1} \breve{\mathbf x}
+
\sqrt{\frac{\breve{\mathbf x}^T \bSigma^{-1} \breve{\mathbf x}}{n_1+n_2-p} }
\left( \mathbf L \mathbf R_{\breve{\mathbf x}} \mathbf L^T\right)^{1/2} \mathbf t_0
\right),
\end{eqnarray}
where
$\mathbf R_{\breve{ \mathbf x}} = \mathbf \Sigma^{-1} - \mathbf \Sigma^{-1} \breve{ \mathbf x} \breve{ \mathbf x}^{T} \mathbf \Sigma^{-1} / \breve{ \mathbf x}^{T} \mathbf \Sigma^{-1} \breve{ \mathbf x}$;
$\xi \sim \chi^2_{n_1+n_2-p-1}$, $\breve{\mathbf x} \sim \mathcal N_p \left(\bmu_1 - \bmu_2, \left( \frac{1}{n_1}+ \frac{1}{n_2}\right) \bSigma \right)$, and $\mathbf t_0 \sim t _k (n_1+n_2-p, \mathbf 0_k, \mathbf I_k)$.
Moreover, $\xi$, $\breve {\mathbf x}$ and $\mathbf t_0$ are mutually independent.

\end{theorem}

%---------------------------------
%        Proof
%---------------------------------
\begin{proof}
From Lemma \ref{l1}.(c) and Theorem 3.4.1 of \citet{GN2000} we obtain that
\begin{eqnarray}
\frac{1}{n_1 +n_2 -2} \mathbf S_{pl}^{-1}
\sim
\mathcal{IW}_p (n_1+n_2+p-1, \bSigma^{-1}).
\end{eqnarray}
Also, since $\breve {\mathbf x} =\bar{\mathbf x }^{(1)} - \bar{\mathbf x }^{(2)}$ and $\mathbf S_{pl}$ are independent,
 the conditional distribution of\\ $\hat \btheta = \mathbf L \mathbf S^{-1}_{pl} \breve {\mathbf x}$ given $\breve {\mathbf x} = \breve{ \mathbf x}^* $
equals to the distribution of $\btheta^*=\mathbf L \mathbf S^{-1}_{pl} \breve{ \mathbf x}^*$
and it can be rewritten in the following form
\begin{eqnarray*}
\btheta^*
&
{\buildrel d \over =}
&
(n_1+n_2 - 2) \breve{ \mathbf x}^{*T} \mathbf \Sigma^{-1} \breve{ \mathbf x}^*
\frac{\mathbf L \mathbf S^{-1}_{pl} \breve{ \mathbf x}^*} {\breve{ \mathbf x}^{*T} \mathbf S^{-1}_{pl} \breve{ \mathbf x}^*}
\frac{\breve{ \mathbf x}^{*T} \mathbf S^{-1}_{pl} \breve{ \mathbf x}^*}{(n_1+n_2-2) \breve{ \mathbf x}^{*T} \mathbf \Sigma^{-1} \breve{ \mathbf x}^*}.
\end{eqnarray*}

Applying Theorem 3.2.12 of \citet{M1982} we obtain that
\begin{eqnarray}
\labmarg{pth2eq1}
\xi^* = (n_1 + n_2 - 2) \frac{\breve{ \mathbf x}^{*T} \mathbf \Sigma^{-1} \breve{ \mathbf x}^*}{\breve{ \mathbf x}^{*T} \mathbf S^{-1}_{pl} \breve{ \mathbf x}^*}
\sim
\chi^2_{n_1+n_2 - p - 1}
\end{eqnarray}
and its distribution is independent of $ \breve{ \mathbf x}^*$. Hence,
\begin{eqnarray}
\labmarg{pth2eq1a}
\xi = (n_1 + n_2 - 2) \frac{\breve{ \mathbf x}^{T} \mathbf \Sigma^{-1} \breve{ \mathbf x}}{\breve{ \mathbf x}^{T} \mathbf S^{-1}_{pl} \breve{ \mathbf x}}
\sim
\chi^2_{n_1+n_2 - p - 1}
\end{eqnarray}
and $\xi$, $\breve{ \mathbf x}$ are independent.

Using Theorem 3 of \citet{BO2008} we get that $ \breve{ \mathbf x}^{*T} \mathbf S^{-1}_{pl}  \breve{ \mathbf x}^* $ is independent of $\mathbf L \mathbf S^{-1}_{pl}  \breve{ \mathbf x}^* /  \breve{ \mathbf x}^{*T} \mathbf S^{-1}_{pl}  \breve{ \mathbf x}^*$ for given $ \breve{ \mathbf x}^*$.
Therefore, $\xi^*$ is  independent of $ \breve{ \mathbf x}^{*T} \mathbf \Sigma^{-1}  \breve{ \mathbf x}^* \cdot \mathbf L \mathbf S^{-1}_{pl}  \breve{ \mathbf x}^* /  \breve{ \mathbf x}^{*T} \mathbf S^{-1} _{pl} \breve{ \mathbf x}^*$
and, respectively, $\xi$ is independent of
$ \breve{ \mathbf x}^T \mathbf \Sigma^{-1}  \breve{ \mathbf x} \cdot \mathbf L \mathbf S^{-1}_{pl}  \breve{ \mathbf x} /  \breve{ \mathbf x}^{T} \mathbf S^{-1} _{pl} \breve{ \mathbf x}$.
Furthermore, from the proof of Theorem 1 of \citet{BS2008} it holds that
\begin{eqnarray}
\labmarg{pth2eq2}
  \breve{ \mathbf x}^{*T} \mathbf \Sigma^{-1} \breve{ \mathbf x}^*
\frac{\mathbf L \mathbf S^{-1}_{pl} \breve{ \mathbf x}^*} {\breve{ \mathbf x}^{*T} \mathbf S^{-1}_{pl} \breve{ \mathbf x}^*}
\sim
t _k
\left(
n_1+n_2 - p ; \mathbf L \mathbf \Sigma^{-1} \breve{ \mathbf x}^*,
\frac{\breve{ \mathbf x}^{*T} \mathbf \Sigma^{-1} \breve{ \mathbf x}^*}{n_1+n_2-p}
\mathbf L \mathbf R_{\breve{ \mathbf x}^*}  \mathbf L^T
\right)
\end{eqnarray}
with
$\mathbf R_{\breve{ \mathbf x}^*} = \mathbf \Sigma^{-1} - \mathbf \Sigma^{-1} \breve{ \mathbf x}^* \breve{ \mathbf x}^{*T} \mathbf \Sigma^{-1} / \breve{ \mathbf x}^{*T} \mathbf \Sigma^{-1} \breve{ \mathbf x}^*$.

Thus, we obtain the following stochastic representation of $\hat \btheta$ which is given by
\begin{eqnarray}
\hat \btheta
&
{\buildrel d \over =}
&
(n_1+n_2-2)
\xi^{-1}
\left(
\mathbf L \bSigma^{-1} \breve{\mathbf x}
+
\sqrt{\frac{\breve{\mathbf x}^T \bSigma^{-1} \breve{\mathbf x}}{n_1+n_2-p} }
\left( \mathbf L \mathbf R_{\breve{\mathbf x}} \mathbf L^T\right)^{1/2} \mathbf t_0
\right),
\end{eqnarray}
where
$\mathbf R_{\breve{ \mathbf x}} = \mathbf \Sigma^{-1} - \mathbf \Sigma^{-1} \breve{ \mathbf x} \breve{ \mathbf x}^{T} \mathbf \Sigma^{-1} / \breve{ \mathbf x}^{T} \mathbf \Sigma^{-1} \breve{ \mathbf x}$;
$\xi \sim \chi^2_{n_1+n_2-p-1}$, $\breve{\mathbf x} \sim \mathcal N_p \left(\bmu_1 - \bmu_2, \left( \frac{1}{n_1}+ \frac{1}{n_2}\right) \bSigma \right)$, and $\mathbf t_0 \sim t _k (n_1+n_2-p, \mathbf 0_k, \mathbf I_k)$.
Moreover, $\xi$, $\breve {\mathbf x}$ and $\mathbf t_0$ are mutually independent.
The theorem is proved.
\end{proof}

In the next corollary we consider the special case when $k = 1$, that is, when $\mathbf L = \mathbf l^T$ is a $p$-dimensional vector of constants.

%---------------------------------
%        Corollary 1
%---------------------------------
\begin{corollary}
\labmarg{c1}
Let $\lambda = 1/n_{1} + 1/n_{2}$ and let $\mathbf l$ be a $p$-dimensional vector of constants.
Then,  under the condition of Theorem \ref{th1}, the stochastic representation of $\hat \theta = \mathbf l^T \hat{\mathbf a}$ is given by
\begin{eqnarray} \label{stoch_pres_theta}
\hat \theta
&{\buildrel d \over =}
&
(n_1+n_2-2)
\xi^{-1}
\left( \mathbf l^{T}\mathbf\Sigma^{-1}(\boldsymbol{\mu}_{1}-\boldsymbol{\mu}_{2})+\sqrt{\left(\lambda+\frac{\lambda (p-1)}{n_{1}+n_{2}-p}u\right)\mathbf l^{T}\mathbf\Sigma^{-1}\mathbf  l } z_0 \right),
\end{eqnarray}
where
$\xi \sim \chi^2_{n_1+n_2-p-1}$,
$z_0 \sim \mathcal N (0, 1)$,
$u \sim \mathcal F \left(p-1, n_{1}+n_{2}-p, (\boldsymbol{\mu}_{1}-\boldsymbol{\mu}_{2})^{T}\mathbf R_{\mathbf l}(\boldsymbol{\mu}_{1}-\boldsymbol{\mu}_{2})/ \lambda \right)$
(non-central
$\mathcal F$-distribution with $p-1$ and $n_1 + n_2 - p$ degrees of freedom and non-centrality parameter $(\boldsymbol{\mu}_{1}-\boldsymbol{\mu}_{2})^{T}\mathbf R_{\mathbf l}(\boldsymbol{\mu}_{1}-\boldsymbol{\mu}_{2})/ \lambda$)
with
$\mathbf R_{\mathbf l}=\mathbf\Sigma^{-1}-\mathbf \Sigma^{-1}\mathbf l \mathbf l^{T}\mathbf \Sigma^{-1}/ \mathbf l^{T}\mathbf\Sigma^{-1}\mathbf l$;
$\xi$, $z_0$ and $u$ are mutually independently distributed.
\end{corollary}

%---------------------------------
%        Proof
%---------------------------------
\begin{proof}
From Theorem \ref{th1} we get that
\begin{eqnarray}
\hat \theta
&
{\buildrel d \over =}
&
(n_1+n_2-2)
\xi^{-1}
\left(
\mathbf l^T \bSigma^{-1} \breve{\mathbf x}
+
 t_0
\sqrt{\frac{\breve{\mathbf x}^T \bSigma^{-1} \breve{\mathbf x}}{n_1+n_2-p} \cdot  \mathbf l^T \mathbf R_{\breve{\mathbf x}} \mathbf l}
\right)
\\
&
=
&
(n_1+n_2-2)
\xi^{-1}
\left(
\mathbf l^T \bSigma^{-1} \breve{\mathbf x}
+
\frac{t_0}{\sqrt{n_1+n_2-p}}
\sqrt{\mathbf l^T \bSigma^{-1}  \mathbf l }
\sqrt{\breve{\mathbf x}^T \mathbf R_{\mathbf l} \breve{\mathbf x} }
\right),
\end{eqnarray}
where
$\mathbf R_{\mathbf l} = \mathbf \Sigma^{-1} - \mathbf \Sigma^{-1}  \mathbf l \mathbf l^{T} \mathbf \Sigma^{-1} /  \mathbf l^{T} \mathbf \Sigma^{-1}  \mathbf l$;
$\xi \sim \chi^2_{n_1+n_2-p-1}$, $ t_0 \sim t  (n_1+n_2-p, 0, 1)$, and $\breve{\mathbf x} \sim \mathcal N_p \left(\bmu_1 - \bmu_2, \lambda \bSigma \right)$ with $ \lambda =1/n_1+ 1/n_2$;
$\xi$, $ t_0$ and $\breve {\mathbf x}$ are mutually independent.

Because $\breve{\mathbf x} \sim \mathcal N_p \left(\bmu_1 - \bmu_2, \lambda \bSigma \right)$, $\mathbf R_{\mathbf l} \bSigma \mathbf R_{\mathbf l} = \mathbf R_{\mathbf l}$, and $\mathrm {tr} \left[\mathbf R_{\mathbf l} \bSigma \right] = p-1$, the application of Corollary 5.1.3a of \citet{MP1992} leads to
\begin{eqnarray}
\zeta=\lambda^{-1}\breve{\mathbf x}^T \mathbf R_{\mathbf l} \breve{\mathbf x}
\sim \chi^2_{p-1} \left(  \delta^2 \right)
\end{eqnarray}
where $\delta^2 = (\bmu_1 - \bmu_2)^T \mathbf R_{\mathbf l} (\bmu_1-\bmu_2)/ \lambda$.
Moreover, since $\mathbf R_{\mathbf l} \bSigma \bSigma^{-1} \mathbf l= \mathbf 0$, the application of Theorem 5.5.1 of \citet{MP1992} proves that $\mathbf l^T \bSigma^{-1} \breve{\mathbf x}$ and $\zeta$
are independently distributed.

Finally, we note that the random variable $t_0 \sim t (n_1+n_2-p, 0,1) $ has the following
stochastic representation
\begin{eqnarray}
t_0
{\buildrel d \over =}
z_0 \sqrt{\frac{n_1+n_2-p}{ w}},
\end{eqnarray}
where $z_0 \sim \mathcal N (0,1)$ and $w \sim \chi^2_{n_1+n_2-p}$; $z_0$ and $w$ are independent.
Hence,
\begin{eqnarray}
\mathbf l^T\bSigma^{-1} \breve{\mathbf x} +
t_0
\sqrt{\frac{\lambda\zeta \cdot \mathbf l^T \bSigma^{-1} \mathbf l}{n_1+n_2-p} } \Bigg| \zeta,  w
&\sim&
\mathcal N \left( \mathbf l^T \bSigma^{-1} \bmu, \lambda \mathbf l^T \bSigma^{-1} \mathbf l \left( 1+\frac{\zeta}{w}\right) \right)
\\
&=&
\mathcal N
\left(
\mathbf l^T \bSigma^{-1} \bmu,
\lambda \mathbf l^T \bSigma^{-1} \mathbf l
\left(
1 + \frac{p-1}{n_1+n_2-p} u
\right)
\right),
\end{eqnarray}
where
\begin{eqnarray}
u
=
\frac{\zeta / (p-1)} { w/ (n_1+n_2-p)}
\sim
\mathcal F \left( p-1, n_1+n_2-p, (\bmu_1-\bmu_2)^T \mathbf R_{\mathbf l}(\bmu_1-\bmu_2 )/\lambda\right).
\end{eqnarray}
Putting all above together we get the statement of the corollary.
\end{proof}

\subsection{Test for the population discriminant function coefficients}\labmarg{s2_2}
One of the most important questions when the discriminant analysis is performed is to decide which coefficients are the most influential in the decision. Several methods exist in the literature with the following three approaches to be the most popular (c.f., \cite[Section 5.5]{RC2012}): (i) standardized coefficients; (ii) partial $F$-values; (iii) correlations between the variables and the discriminant function. \cite[Theorem 5.7A]{Rencher1998} argued that each of this three methods has several drawbacks. For instance, the correlations between the variables and the discriminant function do not show the multivariate contribution of each variable, but provide only univariate information how each variable separates the groups, ignoring the presence of other variables.

In this section, we propose an alternative approach based on the statistical hypothesis test. Namely, exact statistical tests will be derived on the null hypothesis that two population discriminant function coefficients are equal (two-sided test) as well as on the alternative hypothesis that a coefficient in the discriminant function is larger than another one (one-sided test). The testing hypothesis for the equality of the $i$-th and the $j$-th coefficients in the population discriminant function is given by
\begin{equation}\label{hyp_two-sided}
H_0:~ a_i=a_j \quad \text{against} \quad H_1:~ a_i \neq a_j\,,
\end{equation}
while in the case of one-sided test we check if
\begin{equation}\label{hyp_one-sided}
H_0:~ a_i\le a_j \quad \text{against} \quad H_1:~ a_i > a_j \,.
\end{equation}

In both cases the following test statistic is suggested
\begin{equation}\label{test_stat}
T=\sqrt{n_1+n_2-p-1}\frac{
\bl^T \mathbf S_{pl}^{-1}(\bar{\mathbf x}^{(1)}-\bar{\mathbf x}^{(2)})}{\sqrt{\bl^T\mathbf S_{pl}^{-1}\bl}
\sqrt{(n_1+n_2-2)(\frac{1}{n_1}+\frac{1}{n_2})+(\bar{\mathbf x}^{(1)}-\bar{\mathbf x}^{(2)})^T
\hat{\mathbf R}_{\bl}(\bar{\mathbf x}^{(1)}-\bar{\mathbf x}^{(2)})}}
\end{equation}
with
$$\hat{\mathbf R}_{\bl}=\mathbf S_{pl}^{-1}-\frac{\mathbf S_{pl}^{-1}\bl\bl^\top\mathbf S_{pl}^{-1}}{\bl^\top\mathbf S_{pl}^{-1}\bl} \quad\text{and} \quad \mathbf l=(0,..,0,\underbrace{1}_i,0...,0,\underbrace{-1}_j,0,...,0)^\top.$$
The distribution of $T$ follows from \cite[Theorem 6]{BOK2011} and it is summarized in Theorem \ref{th1_2}.

\begin{theorem} \label{th1_2}
Let $\lambda = 1/n_{1} + 1/n_{2}$ and let $\mathbf l$ be a $p$-dimensional vector of constants.
Then,  under the condition of Theorem \ref{th1},
\begin{enumerate}[(a)]
\item the density of $T$ is given by
\begin{equation}\label{den_TSR}
f_{T} (x)  =\frac{n_1+n_2-p}{\lambda(p-1)}\int_0^{\infty}
f_{t_{n_1+n_2-p-1,\delta_1(y)}}(x)f_{\mathcal{F}_{p-1,n_1+n_2-p,s/\lambda}}\left(\frac{n_1+n_2-p}{\lambda(p-1)}y\right)dy
\end{equation}
with $\delta_1(y)=\eta/\sqrt{\lambda+y}$, $\eta=\frac{\bl^T
\bSigma^{-1}(\bmu_1-\bmu_2)}{\sqrt{\bl^T\bSigma^{-1}\bl}}$, and $s=(\bmu_1-\bmu_2)^T \mathbf{R}_{\bl} (\bmu_1-\bmu_2)$; the symbol $f_{G}(.)$ denotes the density of the distribution $G$.

\item Under the null hypothesis it holds that $T \sim
t_{n_1+n_2-p-1}$ and  $T$ is independent of $(\bar{\mathbf x}^{(1)}-\bar{\mathbf x}^{(2)})^T
\hat{\mathbf R}_{\bl}(\bar{\mathbf x}^{(1)}-\bar{\mathbf x}^{(2)})$.
\end{enumerate}
\end{theorem}

Theorem \ref{th1_2} shows that the test statistics $T$ has a standard $t$-distribution under the null hypothesis. As a result, the suggested test will reject the null hypothesis of the two-sided test \eqref{hyp_two-sided} as soon as $|T|>t_{n_1+n_2-p-1;1-\alpha/2}$.

The situation is more complicated in the case of the one-sided test \eqref{hyp_one-sided}. In this case the maximal probability of the type I error has to be control. For that reason, we first calculate the probability of rejection of the null hypothesis for all possible parameter values and after that we calculate its maximum for the parameters which correspond to the null hypothesis in \eqref{hyp_one-sided}. Since the distribution of $T$ depends on $\bmu_1$, $\bmu_2$, and $\bSigma$ only over $\eta$ and $s$ (see, Theorem \ref{th1_2}), the task of finding the maximum is significantly simplified. Let $F_G(.)$ denotes the distribution function of the distribution $G$. For any constant $q$, we get
\begin{eqnarray*}
\mathbb{P}(T>q)&=& \int_{q}^{+\infty} f_T(x)dx\\
&=& \int_{q}^{+\infty}\frac{n_1+n_2-p}{\lambda(p-1)}\int_0^{\infty}
f_{t_{n_1+n_2-p-1,\delta_1(y)}}(x)f_{\mathcal{F}_{p-1,n_1+n_2-p,s/\lambda}}(\frac{n_1+n_2-p}{\lambda(p-1)}y)dy dx\\
&=&\frac{n_1+n_2-p}{\lambda(p-1)}\int_0^{\infty}f_{\mathcal{F}_{p-1,n_1+n_2-p,s/\lambda}}\left(\frac{n_1+n_2-p}{\lambda(p-1)}y\right) \int_{q}^{+\infty}
f_{t_{n_1+n_2-p-1,\delta_1(y)}}(x) dx dy\\
&=&\frac{n_1+n_2-p}{\lambda(p-1)} \int_0^{\infty} (1-F_{t_{n_1+n_2-p-1,\delta_1(y)}}(q)) f_{\mathcal{F}_{p-1,n_1+n_2-p,s/\lambda}}\left(\frac{n_1+n_2-p}{\lambda(p-1)}y\right) dy\\
&\le& \frac{n_1+n_2-p}{\lambda(p-1)} \int_0^{\infty} (1-F_{t_{n_1+n_2-p-1,0}}(q)) f_{\mathcal{F}_{p-1,n_1+n_2-p,s/\lambda}}\left(\frac{n_1+n_2-p}{\lambda(p-1)}y\right) dy\\
&=& (1-F_{t_{n_1+n_2-p-1,0}}(q)).
\end{eqnarray*}
where the last equality follows from the fact that the distribution function of the non-central $t$-distribution is a decreasing function in non-centrality parameter and $\delta_1(y)\le 0$. Consequently, we get $q=t_{n_1+n_2-p-1;1-\alpha}$ and the one-sided test rejects the null hypothesis in \eqref{hyp_one-sided} as soon as $T>t_{n_1+n_2-p-1;1-\alpha}$.

\subsection{Classification analysis}\labmarg{s2_3}

Having a new observation vector $\bx$, we classify it to one of the considered two groups. Assuming that no prior information is available about the classification result, i.e. the prior probability of each group is $1/2$, the decision which is based on the optimal rule is to assign the observation vector $\mathbf{x}$ to the first group as soon as the following inequality holds (c.f., \cite[Section 6.2]{Rencher1998})
\begin{equation}\label{eq_true}
(\bmu_1-\bmu_2)^\top \bSigma^{-1} \bx >\frac{1}{2} (\bmu_1-\bmu_2)^\top \bSigma^{-1}(\bmu_1+\bmu_2)
\end{equation}
and to the second group otherwise. The error rate is defined as the probability of classifying the observation $\bx$ into one group, while it comes from another one. \cite{Rencher1998} presented the expression of the error rate expressed as
\begin{eqnarray*}
ER_p(\Delta)&=&\frac{1}{2}\mathbb{P}(\text{classify to the first group $|$ second group is true})\\
&+&\frac{1}{2}\mathbb{P}(\text{classify to the second group $|$ first group is true})
\nonumber\\
&=&\Phi\left(-\frac{\Delta}{2}\right) ~~ \text{with} ~~\Delta^2=(\bmu_1-\bmu_2)^\top \bSigma^{-1}(\bmu_1-\bmu_2)\,,
\end{eqnarray*}
where $\Phi(.)$ denotes the distribution function of the standard normal distribution.

In practice, however, $\bmu_1$, $\bmu_2$, and $\bSigma$ are unknown quantities and the decision is based on the inequality
\begin{equation}\label{eq_est}
(\bar{\bx}^{(1)}-\bar{\bx}^{(2)})^\top \bS_{pl}^{-1} \bx >\frac{1}{2} (\bar{\bx}^{(1)}-\bar{\bx}^{(2)})^\top \bS_{pl}^{-1}(\bar{\bx}^{(1)}+\bar{\bx}^{(2)})
\end{equation}
instead. Next, we derive the error rate of the decision rule \eqref{eq_est}. Let
\begin{eqnarray}\label{hat_d}
\hat d&=& (\bar{\bx}^{(1)}-\bar{\bx}^{(2)})^\top \bS_{pl}^{-1} \bx -\frac{1}{2} (\bar{\bx}^{(1)}-\bar{\bx}^{(2)})^\top \bS_{pl}^{-1}(\bar{\bx}^{(1)}+\bar{\bx}^{(2)})\nonumber\\
&=& (\bar{\bx}^{(1)}-\bar{\bx}^{(2)})^\top \bS_{pl}^{-1} \left(\bx-\frac{1}{2}(\bar{\bx}^{(1)}+\bar{\bx}^{(2)})\right)\,.
\end{eqnarray}
In Theorem \ref{th1_3} we present the stochastic representation of $\hat d$.

\begin{theorem}\label{th1_3}
Let $\lambda = 1/n_{1} + 1/n_{2}$. Then, under the condition of Theorem \ref{th1}, the stochastic representation of $\hat d$ is given by
\begin{eqnarray}\label{hat_d_th1_3}
\hat d &{\buildrel d \over =} & \frac{n_1+n_2-2}{\xi}
\Bigg((-1)^{i-1}\frac{\lambda n_i-2}{2\lambda n_i} \left(\lambda \xi_2+(\Delta+\sqrt{\lambda}w_0)^2\right)+\frac{(-1)^{i-1}}{\lambda n_i}\left(\Delta^2+\sqrt{\lambda}\Delta w_0\right)\nonumber\\
&+&\sqrt{\left(1+\frac{1}{n_{1}+n_{2}}+\frac{p-1}{n_{1}+n_{2}-p}u\right)}\sqrt{\lambda \xi_2+(\Delta+\sqrt{\lambda}w_0)^2} z_0 \Bigg)
~~\text{for} ~~ i=1,2,
\end{eqnarray}
where $u|\xi_1,\xi_2,w_0 \sim \mathcal F \left(p-1, n_{1}+n_{2}-p, (n_1+n_2)^{-1}\xi_1 \right)$ with $\xi_1|\xi_2,w_0\sim \chi_{p-1,\delta^2_{\xi_2,w_0}}$ and
$\delta^2_{\xi_2,w_0}=\frac{n_1n_2}{n_i^2}\frac{\Delta^2\xi_2}{\lambda \xi_2+(\Delta+\sqrt{\lambda}w_0)^2}$, $z_0,w_0 \sim \mathcal N (0, 1)$, $\xi \sim \chi^2_{n_1+n_2-p-1}$, $\xi_2\sim\chi^2_{p-1}$; $\xi$, $z_0$ are independent of $u,\xi_1,\xi_2,w_0$ where $\xi_2$ and $w_0$ are independent as well.
\end{theorem}

%---------------------------------
%        Proof
%---------------------------------
\begin{proof}
Let $\bx \sim \mathcal{N}_p\left(\bmu_i,\bSigma\right)$,
Since $\bar{\bx}^{(1)}$, $\bar{\bx}^{(2)}$, $\bx$, and $\bS_{pl}$ are independently distributed, we get that the conditional distribution of $\hat d$ given $\bar{\bx}^{(1)}=\bx^{(1)}_0$ and $\bar{\bx}^{(2)}=\bx^{(2)}_0$ is equal to the distribution of $d_0$ defined by
\[d_0=(\bar{\bx}^{(1)}_0-\bar{\bx}^{(2)}_0)^\top \bS_{pl}^{-1} \tilde{\bx}\,,\]
where $\tilde{\bx}=\bx-\frac{1}{2}(\bar{\bx}^{(1)}_0+\bar{\bx}^{(2)}_0)\sim \mathcal{N}_p\left(\bmu_i-\frac{1}{2}(\bar{\bx}^{(1)}_0+\bar{\bx}^{(2)}_0),\bSigma\right)$, $(n_1+n_2-2)\bS_{pl}\sim \mathcal{W}_p(n_1+n_2-2,\bSigma)$, $\tilde{\bx}$ and $\bS_{pl}$ are independent.

Following the proof of Corollary \ref{c1}, we get
\begin{eqnarray*}
d_0&{\buildrel d \over =} & (n_1+n_2-2)\xi^{-1}
\Bigg( (\bar{\bx}^{(1)}_0-\bar{\bx}^{(2)}_0)^{T}\mathbf\Sigma^{-1}\left(\bmu_i-\frac{1}{2}(\bar{\bx}^{(1)}_0+\bar{\bx}^{(2)}_0)\right)\\
&+&\sqrt{\left(1+\frac{(p-1)}{n_{1}+n_{2}-p}u\right)(\bar{\bx}^{(1)}_0-\bar{\bx}^{(2)}_0)^{T}\mathbf\Sigma^{-1}(\bar{\bx}^{(1)}_0-\bar{\bx}^{(2)}_0)} z_0 \Bigg),
\end{eqnarray*}
where $u \sim \mathcal F \left(p-1, n_{1}+n_{2}-p, \left(\bmu_i-\frac{1}{2}(\bar{\bx}^{(1)}_0+\bar{\bx}^{(2)}_0)\right)^{T}\mathbf R_{0}\left(\bmu_i-\frac{1}{2}(\bar{\bx}^{(1)}_0+\bar{\bx}^{(2)}_0)\right) \right)$
with $\mathbf R_{0}$ $=\mathbf\Sigma^{-1}-\mathbf \Sigma^{-1}(\bar{\bx}^{(1)}_0-\bar{\bx}^{(2)}_0) (\bar{\bx}^{(1)}_0-\bar{\bx}^{(2)}_0)^{T}\mathbf \Sigma^{-1}/ (\bar{\bx}^{(1)}_0-\bar{\bx}^{(2)}_0)^{T}\mathbf\Sigma^{-1}(\bar{\bx}^{(1)}_0-\bar{\bx}^{(2)}_0)$, $z_0 \sim \mathcal N (0, 1)$, and $\xi \sim \chi^2_{n_1+n_2-p-1}$;
$\xi$, $z_0$ and $u$ are mutually independently distributed.

In using that
\[\bmu_i-\frac{1}{2}(\bar{\bx}^{(1)}_0+\bar{\bx}^{(2)}_0)=\bmu_i-\bar{\bx}^{(i)}_0+(-1)^{i-1} \frac{1}{2}(\bar{\bx}^{(1)}_0-\bar{\bx}^{(2)}_0)\]
and $(\bar{\bx}^{(1)}_0-\bar{\bx}^{(2)}_0)^{T}\mathbf{R}_0=\mathbf{0}$, we get
\begin{eqnarray*}
\hat d &{\buildrel d \over =} & \frac{n_1+n_2-2}{\xi}
\Bigg(\frac{(-1)^{i-1}}{2} (\bar{\bx}^{(1)}-\bar{\bx}^{(2)})^{T}\mathbf\Sigma^{-1}(\bar{\bx}^{(1)}-\bar{\bx}^{(2)})
-(\bar{\bx}^{(1)}-\bar{\bx}^{(2)})^{T}\mathbf\Sigma^{-1}\left(\bar{\bx}^{(i)}-\bmu_i\right)\\
&+&\sqrt{\left(1+\frac{p-1}{n_{1}+n_{2}-p}u\right)(\bar{\bx}^{(1)}-\bar{\bx}^{(2)})^{T}\mathbf\Sigma^{-1}(\bar{\bx}^{(1)}-\bar{\bx}^{(2)})} z_0 \Bigg),
\end{eqnarray*}
where $u|\bar{\bx}^{(1)},\bar{\bx}^{(2)} \sim \mathcal F \left(p-1, n_{1}+n_{2}-p, \left(\bar{\bx}^{(i)}-\bmu_i\right)^{T}\mathbf R_{\bx}\left(\bar{\bx}^{(i)}-\bmu_i\right) \right)$
with $\mathbf R_{\bx}$ $=\mathbf\Sigma^{-1}-\mathbf \Sigma^{-1}(\bar{\bx}^{(1)}-\bar{\bx}^{(2)}) (\bar{\bx}^{(1)}-\bar{\bx}^{(2)})^{T}\mathbf \Sigma^{-1}/ (\bar{\bx}^{(1)}-\bar{\bx}^{(2)})^{T}\mathbf\Sigma^{-1}(\bar{\bx}^{(1)}-\bar{\bx}^{(2)})$, $z_0 \sim \mathcal N (0, 1)$, and $\xi \sim \chi^2_{n_1+n_2-p-1}$;
$\xi$, $z_0$ are independent of $u,\bar{\bx}^{(1)},\bar{\bx}^{(2)}$.

Since $\bar{\bx}^{(1)}$ and $\bar{\bx}^{(2)}$ are independent and normally distributed, we get that
\[
\left(
\begin{array}{c}
  \bar{\bx}^{(i)}-\bmu_i \\
  \bar{\bx}^{(1)}-\bar{\bx}^{(2)}
\end{array}
\right)
\sim \mathcal{N}_{2p} \left(
\left(
\begin{array}{c}
  \mathbf{0} \\
  \bmu_1-\bmu_2
\end{array}
\right),
\left(
\begin{array}{cc}
  \frac{1}{n_i}\bSigma & \frac{(-1)^{i-1}}{n_i}\bSigma\\
  \frac{(-1)^{i-1}}{n_i}\bSigma & \lambda \bSigma
\end{array}
\right)\right)
\]
and, consequently,
\[
 \bar{\bx}^{(i)}-\bmu_i | ( \bar{\bx}^{(1)}-\bar{\bx}^{(2)})
\sim \mathcal{N}_{p} \left(
\frac{(-1)^{i-1}}{\lambda n_i}(\bar{\bx}^{(1)}-\bar{\bx}^{(2)}-(\bmu_1-\bmu_2)),\frac{1}{n_1+n_2}\bSigma
\right)\,,
\]
where we used that $\frac{1}{n_i}-\frac{1}{\lambda n_i^2}=\frac{1}{n_1+n_2}$.

The application of Theorem 5.5.1 in \cite{MP1992} shows that given $( \bar{\bx}^{(1)}-\bar{\bx}^{(2)})$ the random variables $(\bar{\bx}^{(1)}-\bar{\bx}^{(2)})^{T}\mathbf\Sigma^{-1}(\bar{\bx}^{(i)}-\bmu_i)$ and $(\bar{\bx}^{(i)}-\bmu_i)\mathbf{R}_{\bx}(\bar{\bx}^{(i)}-\bmu_i)$ are independently distributed with
\begin{eqnarray*}
&&(\bar{\bx}^{(1)}-\bar{\bx}^{(2)})^{T}\mathbf\Sigma^{-1}(\bar{\bx}^{(i)}-\bmu_i)|(\bar{\bx}^{(1)}-\bar{\bx}^{(2)}) \\
&\sim&
\mathcal{N}\left(\frac{(-1)^{i-1}}{\lambda n_i}(\bar{\bx}^{(1)}-\bar{\bx}^{(2)})^{T}\mathbf\Sigma^{-1} (\bar{\bx}^{(1)}-\bar{\bx}^{(2)}-(\bmu_1-\bmu_2)), \frac{1}{n_1+n_2}(\bar{\bx}^{(1)}-\bar{\bx}^{(2)})^{T}\mathbf\Sigma^{-1}(\bar{\bx}^{(1)}-\bar{\bx}^{(2)})\right)
\end{eqnarray*}
and, by using Corollary 5.1.3a of \cite{MP1992},
\[(n_1+n_2)(\bar{\bx}^{(i)}-\bmu_i)^T\mathbf{R}_{\bx}(\bar{\bx}^{(i)}-\bmu_i)|(\bar{\bx}^{(1)}-\bar{\bx}^{(2)})
\sim \chi_{p-1,\delta^2_\bx}\]
with
\begin{eqnarray*}
\delta^2_{\bx}&=&\frac{n_1+n_2}{\lambda^2 n_i^2}(\bar{\bx}^{(1)}-\bar{\bx}^{(2)}-(\bmu_1-\bmu_2))^{T} \mathbf{R}_{\bx} (\bar{\bx}^{(1)}-\bar{\bx}^{(2)}-(\bmu_1-\bmu_2))\\
&=&\frac{n_1+n_2}{\lambda^2 n_i^2}(\bmu_1-\bmu_2)^{T} \mathbf{R}_{\bx} (\bmu_1-\bmu_2)\\
&=&\frac{n_1+n_2}{\lambda^2 n_i^2}\frac{(\bmu_1-\bmu_2)^{T} \bSigma^{-1} (\bmu_1-\bmu_2)}{(\bar{\bx}^{(1)}-\bar{\bx}^{(2)})^{T} \bSigma^{-1} (\bar{\bx}^{(1)}-\bar{\bx}^{(2)})}(\bar{\bx}^{(1)}-\bar{\bx}^{(2)})^{T} \mathbf{R}_{\bmu}(\bar{\bx}^{(1)}-\bar{\bx}^{(2)})
\end{eqnarray*}
where we use that $(\bar{\bx}^{(1)}-\bar{\bx}^{(2)})^{T}\mathbf{R}_{\bx}=\mathbf{0}$ and $\mathbf{R}_{\bmu}= \mathbf\Sigma^{-1}-\mathbf \Sigma^{-1}(\bmu_1-\bmu_2) (\bmu_1-\bmu_2)^{T}\mathbf \Sigma^{-1}/ (\bmu_1-\bmu_2)^{T}\mathbf\Sigma^{-1}(\bmu_1-\bmu_2)$.

As a result, we get
\begin{eqnarray*}
\hat d &{\buildrel d \over =} & \frac{n_1+n_2-2}{\xi}
\Bigg((-1)^{i-1}\frac{\lambda n_i-2}{2\lambda n_i} \Delta_{\bx}^2+\frac{(-1)^{i-1}}{\lambda n_i}(\bmu_1-\bmu_2)^{T}\mathbf\Sigma^{-1}(\bar{\bx}^{(1)}-\bar{\bx}^{(2)})\\
&+&\sqrt{\left(1+\frac{1}{n_{1}+n_{2}}+\frac{p-1}{n_{1}+n_{2}-p}u\right)}\Delta_{\bx} z_0 \Bigg),
\end{eqnarray*}
where $\Delta_{\bx}^2=(\bar{\bx}^{(1)}-\bar{\bx}^{(2)})^{T}\mathbf\Sigma^{-1}(\bar{\bx}^{(1)}-\bar{\bx}^{(2)})$, $u|\bar{\bx}^{(1)},\bar{\bx}^{(2)} \sim \mathcal F \left(p-1, n_{1}+n_{2}-p, (n_1+n_2)^{-1}\xi_1 \right)$ with $\xi_1\sim \chi_{p-1,\delta^2_\bx}$, $z_0 \sim \mathcal N (0, 1)$, and $\xi \sim \chi^2_{n_1+n_2-p-1}$;
$\xi$, $z_0$ are independent of $u,\xi_1,\bar{\bx}^{(1)},\bar{\bx}^{(2)}$.

Finally, it holds with $\Delta^2=(\bmu_1-\bmu_2)^{T}\mathbf\Sigma^{-1}(\bmu_1-\bmu_2)$ that
\begin{eqnarray*}
\Delta_{\bx}^2 &=& (\bar{\bx}^{(1)}-\bar{\bx}^{(2)})^{T} \mathbf{R}_{\bmu}(\bar{\bx}^{(1)}-\bar{\bx}^{(2)})+\frac{\left((\bmu_1-\bmu_2)^{T}\mathbf\Sigma^{-1}(\bar{\bx}^{(1)}-\bar{\bx}^{(2)})\right)^2}{\Delta^2} ,
\end{eqnarray*}
where both summands are independent following Theorem 5.5.1 in \cite{MP1992}. The application of Corollary 5.1.3a in \cite{MP1992} leads to
\[\lambda^{-1}(\bar{\bx}^{(1)}-\bar{\bx}^{(2)})^{T} \mathbf{R}_{\bmu}(\bar{\bx}^{(1)}-\bar{\bx}^{(2)})\sim \chi^2_{p-1}\]
and
\[(\bmu_1-\bmu_2)^{T}\mathbf\Sigma^{-1}(\bar{\bx}^{(1)}-\bar{\bx}^{(2)})\sim \mathcal{N}(\Delta^2,\lambda\Delta^2).\]

From the last statement we get the stochastic representation of $\hat d$ expressed as
\begin{eqnarray*}
\hat d &{\buildrel d \over =} & \frac{n_1+n_2-2}{\xi}
\Bigg((-1)^{i-1}\frac{\lambda n_i-2}{2\lambda n_i} \left(\lambda \xi_2+(\Delta+\sqrt{\lambda}w_0)^2\right)+\frac{(-1)^{i-1}}{\lambda n_i}\left(\Delta^2+\sqrt{\lambda}\Delta w_0\right)\\
&+&\sqrt{\left(1+\frac{1}{n_{1}+n_{2}}+\frac{p-1}{n_{1}+n_{2}-p}u\right)}\sqrt{\lambda \xi_2+(\Delta+\sqrt{\lambda}w_0)^2} z_0 \Bigg),
\end{eqnarray*}
where $u|\xi_1,\xi_2,w_0 \sim \mathcal F \left(p-1, n_{1}+n_{2}-p, (n_1+n_2)^{-1}\xi_1 \right)$ with $\xi_1|\xi_2,w_0\sim \chi_{p-1,\delta^2_{\xi_2,w_0}}$ and
$\delta^2_{\xi_2,w_0}=\frac{n_1+n_2}{\lambda^2 n_i^2}\frac{\Delta^2}{\lambda \xi_2+(\Delta+\sqrt{\lambda}w_0)^2} \lambda\xi_2$, $z_0,w_0 \sim \mathcal N (0, 1)$, $\xi \sim \chi^2_{n_1+n_2-p-1}$, $\xi_2\sim\chi^2_{p-1}$; $\xi$, $z_0$ are independent of $u,\xi_1,\xi_2,w_0$ where $\xi_2$ and $w_0$ are independent as well.
\end{proof}

Theorem \ref{th1_3} shows that the distribution of $\hat d$ is determined by six random variables $\xi,\xi_1,\xi_2$, $z_0,w_0$, and $u$. Moreover, it depends on $\bmu_1,\bmu_2$, and $\bSigma$ only via the quadratic form $\Delta$. As a result, the the error rate based on the decision rule \eqref{eq_est} is a function of $\Delta$ only and it is calculated by
\begin{eqnarray}\label{ERs}
ER_s(\Delta)&=&\frac{1}{2}\mathbb{P}(\text{$\hat d >0|$ second group is true})+\frac{1}{2}\mathbb{P}(\text{$\hat d \le 0|$ first group is true})\,.\nonumber
\end{eqnarray}
The two probabilities in \eqref{ERs} can easily be approximated for all $\Delta$, $p$, $n_1$, and $n_2$ with high precision by applying the results of Theorem \ref{th1_3} via the following simulation study
\begin{enumerate}[(i)]
  \item Fix $\Delta$ and $i \in\{1,2\}$.
  \item Generate four independent random variables $\xi_b\sim \chi^2_{n_1+n_2-p-1}$, $\xi_{2;b}\sim \chi^2_{p-1}$, $z_{0;b} \sim\mathcal N (0, 1)$, and $w_{0;b}\sim\mathcal N (0, 1)$.
  \item Generate $\xi_{1,b} \sim \chi_{p-1,\delta^2_{\xi_2,w_0}}$ with $\delta^2_{\xi_{2,b},w_{0,b}}=\frac{n_1n_2}{n_i^2}\frac{\Delta^2\xi_{2;b}}{\lambda \xi_{2;b}+(\Delta+\sqrt{\lambda}w_{0;b})^2}$.
  \item Generate $u \sim  \mathcal F \left(p-1, n_{1}+n_{2}-p, (n_1+n_2)^{-1}\xi_{1,b} \right)$.
  \item Calculate $\hat d_b^{(i)}$ following the stochastic representation \eqref{hat_d_th1_3} of Theorem \ref{th1_3}.
  \item Repeat steps (ii)-(v) for $b=1,...,B$ leading to the sample $\hat d_{1}^{(i)}$, ..., $\hat d_{B}^{(i)}$.
\end{enumerate}
The procedure has to be performed for both values of $i=1,2$ where for $i=1$ the relative number of events $\{\hat d>0\}$ will approximate the first summand in \eqref{ERs} while for $i=2$ the relative number of events $\{\hat d\le0\}$ will approximate the second summand in \eqref{ERs}.

It is important to note that the difference between the error rates calculated for the two decision rules \eqref{eq_true} ad \eqref{eq_est} could be very large as shown in Figure \ref{fig1} where $ER_p(\Delta)$ and $ER_s(\Delta)$ calculated for several values of $n_1=n_2 \in \{50,100,150,250\}$ with fixed values of $p\in \{10,25,50,75\}$. If $p=10$ we do not observe large differences between $ER_p(\Delta)$ and $ER_s(\Delta)$ computed for different sample sizes. However, this statement does not hold any longer when $p$ becomes comparable to both $n_1$ and $n_2$ as documented for $p=50$ and $p=75$. This case is known in the literature as a large-dimensional asymptotic regime and it is investigated in detail in Section 3.

\begin{figure}
 \centerline{
\includegraphics[width=8.8cm]{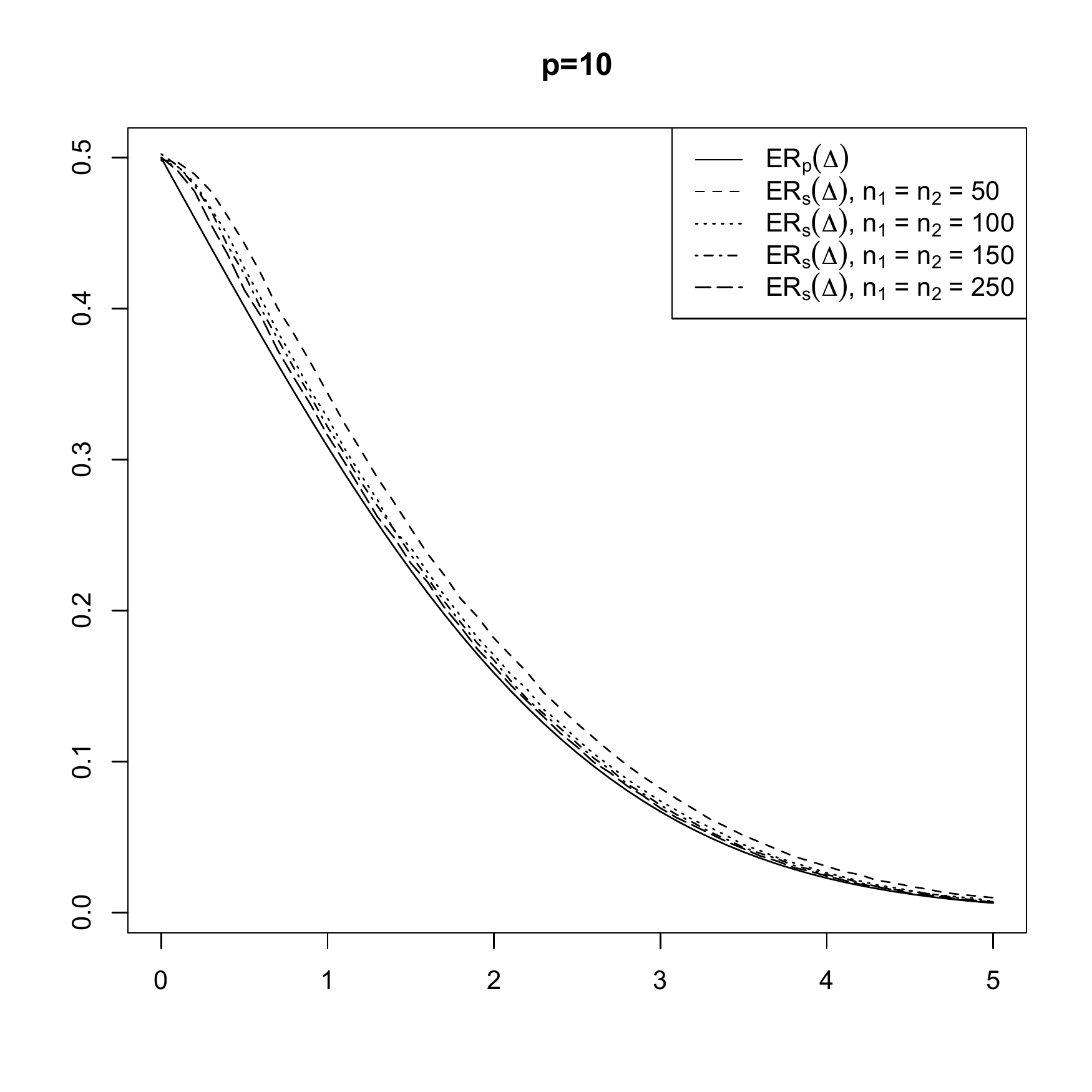}
 \includegraphics[width=8.8cm]{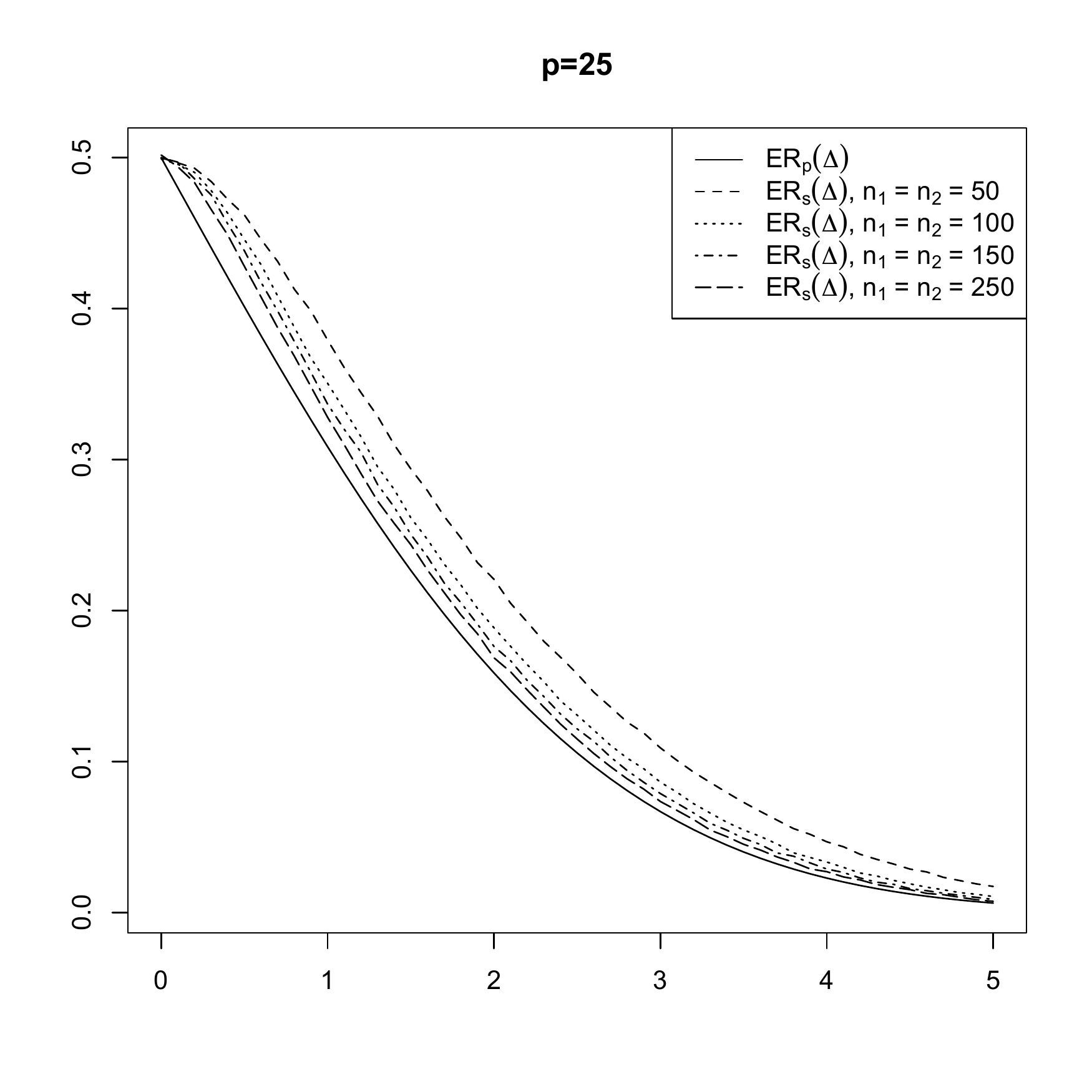}
}
%${ a)\;\; p=50,~n_{1}=25,~n_{2}=475} \hspace{115pt}  { b)\;\; p=250, ~n_{1}=25,~n_{2}=475}$
 \centerline{
\includegraphics[width=8.8cm]{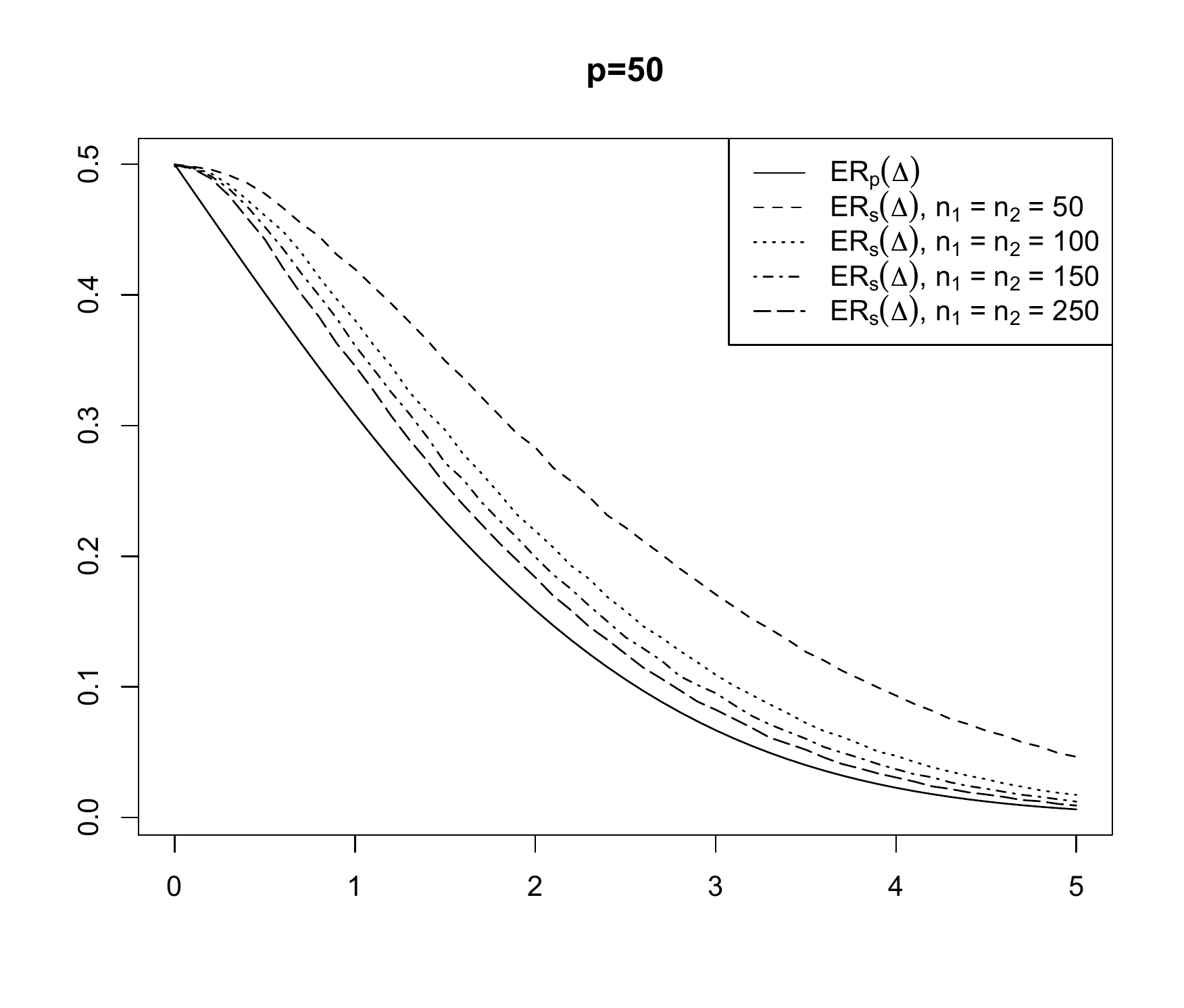}
 \includegraphics[width=8.8cm]{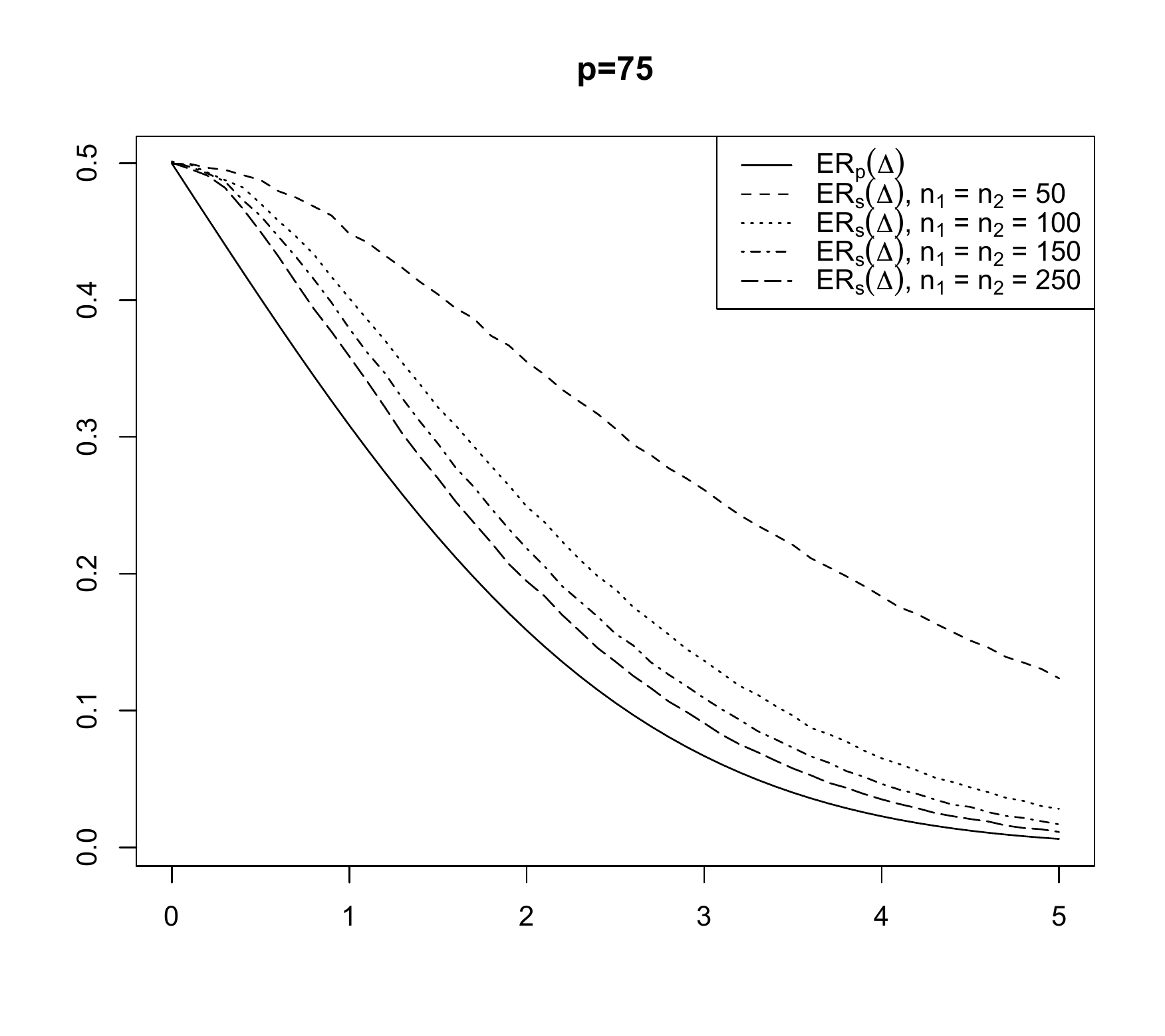}
}
%${ c)\;\; p=400,~n_{1}=25,~n_{2}=475} \hspace{115pt}  { d)\;\; p=475,~n_{1}=25,~n_{2}=475}$
 \caption{Error rates $ER_p(\Delta)$ and $ER_s(\Delta)$ as functions of $\Delta$ for $p\in \{10,25,50,75\}$ and $ER_s(\Delta)$.}
\label{fig1}
 \end{figure}

%%%%%%%%%%%%%%%                            3. Asymptotic Distribution                      %%%%%%%%%%%%%%%%%%%%%
\section{Discriminant analysis under large-dimensional asymptotics}
\labmarg{s4}

In this section we derive the asymptotic distribution of the discriminant function coefficients under the high-dimensional asymptotic regime, that is, when the dimension increases together with the sample sizes and they all tend to infinity. More precisely, we assume that $p/(n_1+n_2) \to c \in [0,1)$ as $n_1+n_2 \to \infty$.

The following conditions are needed for the validity of the asymptotic results:

\begin{enumerate}
\item[(A1)] There exists $\gamma \ge 0$ such that $p^{-\gamma} (\boldsymbol{\mu}_{1}-\boldsymbol{\mu}_{2})^T\bSigma^{-1} (\boldsymbol{\mu}_{1}-\boldsymbol{\mu}_{2})<\infty$ uniformly on $p$.
\item[(A2)]   $0<\lim\limits_{(n_{1}, n_{2}) \to \infty} \left(n_{1}/n_{2}\right)<\infty.$
\end{enumerate}

It is remarkable that, no assumption on the eigenvalues of the covariance matrix $\bSigma$, like they are uniformly bounded on $p$, is imposed. The asymptotic results are also valid when $\bSigma$ possesses unbounded spectrum as well as when its smallest eigenvalue tends to zero as $p \to \infty$. The constant $\gamma$ is a technical one and it controls the growth rate of the quadratic form.
In Theorem \ref{th2} the asymptotic distribution of linear combinations of the discriminant function coefficients is provided.
% The only thing we need to know is if $\gamma =0$ or $\gamma\geq0$.
%\textcolor{blue}

%%%%%%%%%%%%%%                                        Threorem 2                               %%%%%%%%%%%%%%%%%

\begin{theorem}
\labmarg{th2}
Assume (A1) and (A2). Let $\mathbf l$ be a $p$-dimensional vector of constants such that $p^{-\gamma} \mathbf{l}^T\bSigma^{-1} \mathbf{l}<\infty$ is uniformly on $p$, $\gamma \ge 0$. Then, under the conditions of Theorem \ref{th1}, the asymptotic distribution of $\hat \theta = \mathbf l^T \hat{\mathbf a}$ is given by
\begin{align*}
\sqrt{ n_1 + n_2} \sigma_{\gamma}^{-1} \left(\hat \theta
-\frac{1}{1-c} \mathbf l^T \mathbf \Sigma^{-1}(\boldsymbol{\mu}_{1}-\boldsymbol{\mu}_{2})\right)
\stackrel{\mathcal{D}}{\longrightarrow}
\mathcal{N}(0, 1 )
\end{align*}
for $p/(n_1+n_2) \to c \in [0,1)$ as $n_1+n_2 \to \infty$ with
\begin{eqnarray}\label{sig_gam}
\sigma^{2}_{\gamma}
&=&\frac{1}{(1-c)^3}\Bigg(\left(\mathbf l^{T}\mathbf\Sigma^{-1}(\boldsymbol{\mu}_{1}-\boldsymbol{\mu}_{2})\right)^2+\mathbf l^{T}\mathbf\Sigma^{-1}\mathbf l (\boldsymbol{\mu}_{1}-\boldsymbol{\mu}_{2})^{T}\mathbf\Sigma^{-1}(\boldsymbol{\mu}_{1}-\boldsymbol{\mu}_{2})\\
&+&\lambda (n_1+n_2) \mathbf l^{T}\mathbf\Sigma^{-1}\mathbf l \mathds{1}_{\{0\}}(\gamma)
\Bigg) \nonumber
\end{eqnarray}
where $\mathds{1}_{\mathcal{A}}(.)$ denotes the indicator function of set $\mathcal{A}$.
\end{theorem}

%%%%%% Proof Theorem 2   %%%%%%%%%%%%%%%%%

\begin{proof}
Using the stochastic representation \eqref{stoch_pres_theta} of Corollary \ref{c1}, we get
\begin{eqnarray*}
&&\sqrt{ n_1 + n_2} \sigma_{\gamma}^{-1} \left(\hat \theta -\frac{1}{1-c} \mathbf l^T \mathbf \Sigma^{-1}(\boldsymbol{\mu}_{1}-\boldsymbol{\mu}_{2})\right)\\
&{\buildrel d \over =}&\sqrt{ n_1 + n_2} \left((n_1+n_2-2)\xi^{-1}-\frac{1}{1-c} \right) \frac{p^{-\gamma}\mathbf l^{T}\mathbf\Sigma^{-1}(\boldsymbol{\mu}_{1}-\boldsymbol{\mu}_{2})}{p^{-\gamma}\sigma_{\gamma}}\\
&+& \sqrt{ \lambda (n_1 + n_2)} \frac{n_1+n_2-2}{\xi} \sqrt{\left(p^{-\gamma}+p^{-\gamma}\frac{p-1}{n_{1}+n_{2}-p}u\right)} \frac{\sqrt{p^{-\gamma}\mathbf l^{T}\mathbf\Sigma^{-1}\mathbf  l}}{p^{-\gamma}\sigma_{\gamma}} z_0 ,
\end{eqnarray*}
where
$\xi \sim \chi^2_{n_1+n_2-p-1}$, $z_0 \sim \mathcal N (0, 1)$,
$u \sim \mathcal F \left(p-1, n_{1}+n_{2}-p, (\boldsymbol{\mu}_{1}-\boldsymbol{\mu}_{2})^{T}\mathbf R_{\mathbf l}(\boldsymbol{\mu}_{1}-\boldsymbol{\mu}_{2})/ \lambda \right)$
with $\mathbf R_{\mathbf l}=\mathbf\Sigma^{-1}-\mathbf \Sigma^{-1}\mathbf l \mathbf l^{T}\mathbf \Sigma^{-1}/ \mathbf l^{T}\mathbf\Sigma^{-1}\mathbf l$;
$\xi$, $z_0$ and $u$ are mutually independently distributed.

Since, $\xi \sim \chi^2_{n_1+n_2-p-1}$, we get that
\[\sqrt{n_1+n_2-p-1}\left(\frac{\xi}{n_1+n_2-p-1}-1\right) \stackrel{\mathcal{D}}{\longrightarrow}
\mathcal{N}(0, 2 )
\]
for $p/(n_1+n_2) \to c \in [0,1)$ as $n_1+n_2 \to \infty$ and, consequently,
\begin{eqnarray*}
&&\sqrt{ n_1 + n_2} \left((n_1+n_2-2)\xi^{-1}-\frac{1}{1-c} \right) =\frac{\sqrt{ n_1 + n_2}}{\sqrt{n_1+n_2-p-1}}\frac{n_1+n_2-p-1}{\xi}\frac{1}{1-c}\\
&\times&\sqrt{n_1+n_2-p-1} \left((1-c)\frac{n_1+n_2-2}{n_1+n_2-p-1}- \frac{\xi}{n_1+n_2-p-1} \right)\stackrel{\mathcal{D}}{\longrightarrow} \tilde{z_0}\sim
\mathcal{N}\left(0, \frac{2}{1-c}\right)
\end{eqnarray*}
for $\frac{p}{n_1+n_2}=c+o((n_1+n_2)^{-1/2})$ where $z_0$ and $\tilde{z}_0$ are independent.

Furthermore, we get (see, \cite[Lemma 3]{bodnar2016exact})
\begin{eqnarray*}
p^{-\gamma}+p^{-\gamma}\frac{p-1}{n_{1}+n_{2}-p}u -\mathds{1}_{\{0\}}(\gamma)-\frac{c}{1-c}\left(\mathds{1}_{\{0\}}(\gamma)+\frac{p^{-\gamma}(\boldsymbol{\mu}_{1}-\boldsymbol{\mu}_{2})^{T}\mathbf R_{\mathbf l}(\boldsymbol{\mu}_{1}-\boldsymbol{\mu}_{2})}{c\lambda (n_1+n_2)}\right)&\stackrel{a.s.}{\longrightarrow} &0
\end{eqnarray*}

Putting the above results together, we get the statement of the theorem with
\begin{eqnarray*}
\sigma_{\gamma}^2&=&\frac{1}{(1-c)^3}\Bigg(2\left(\mathbf l^{T}\mathbf\Sigma^{-1}(\boldsymbol{\mu}_{1}-\boldsymbol{\mu}_{2})\right)^2+ \mathbf l^{T}\mathbf\Sigma^{-1}\mathbf l (\boldsymbol{\mu}_{1}-\boldsymbol{\mu}_{2})^{T}\mathbf R_{\mathbf l}(\boldsymbol{\mu}_{1}-\boldsymbol{\mu}_{2})\\
&+& \lambda (n_1+n_2) \mathbf l^{T}\mathbf\Sigma^{-1}\mathbf l \mathds{1}_{\{0\}}(\gamma)\Bigg)\\
&=&\frac{1}{(1-c)^3}\Bigg(\left(\mathbf l^{T}\mathbf\Sigma^{-1}(\boldsymbol{\mu}_{1}-\boldsymbol{\mu}_{2})\right)^2+\mathbf l^{T}\mathbf\Sigma^{-1}\mathbf l (\boldsymbol{\mu}_{1}-\boldsymbol{\mu}_{2})^{T}\mathbf\Sigma^{-1}(\boldsymbol{\mu}_{1}-\boldsymbol{\mu}_{2})\\
&+&\lambda (n_1+n_2) \mathbf l^{T}\mathbf\Sigma^{-1}\mathbf l \mathds{1}_{\{0\}}(\gamma)
\Bigg)
\end{eqnarray*}
\end{proof}

The results of Theorem \ref{th2} show that the quantity $\gamma$ is present only in the asymptotic variance $\sigma_{\gamma}^2$. Moreover, if $\gamma>0$, then the factor $\lambda (n_{1}+n_{2})$ vanishes and therefore the assumption (A2) is no longer needed. However, in the case of $\gamma=0$ we need (A2) in order to keep the variance bounded. We further investigate this point via simulations in Section 3.3, by choosing $\gamma>0$ and considering small $n_{1}$ and large $n_{2}$ such that $n_{1}/n_{2}\to 0$.

\subsection{Classification analysis in high dimension}

The error rate of the classification analysis based on the optimal decision rule \eqref{eq_true} remains the same independently of $p$ and it is always equal to
\[ER_p(\Delta)=\Phi\left(-\frac{\Delta}{2}\right) ~~ \text{with} ~~\Delta^2=(\bmu_1-\bmu_2)^\top \bSigma^{-1}(\bmu_1-\bmu_2)\,.\]
In practice, however, $\bmu_1$, $\bmu_2$, and $\bSigma$ are not known and, consequently, one has to make the decision based on \eqref{eq_est} instead of \eqref{eq_true}. In Theorem \ref{th3}, we derived the asymptotic distribution of $\hat{d}$ under the large-dimensional asymptotics.

\begin{theorem}
\labmarg{th3}
Assume (A1) and (A2). Let $p^{-\gamma}\Delta^2 \to \widetilde{\Delta}^2$ and $\lambda n_i \to b_i$ for $p/(n_1+n_2) \to c \in [0,1)$ as $n_1+n_2 \to \infty$. Then, under the conditions of Theorem \ref{th1}, it holds that
\begin{align*}
&p^{\min(\gamma,1)/2}\left(\frac{\hat d}{p^\gamma}- \frac{n_1+n_2-2}{n_1+n_2-p-1}\frac{(-1)^{i-1}}{2}p^{-\gamma}\Delta^2\right)\\
&\stackrel{\mathcal{D}}{\longrightarrow}
\mathcal{N}\Bigg((-1)^{i-1}\frac{c}{1-c}\frac{b_i-2}{2b_i}(b_1+b_2)\mathds{1}_{\{0\}}(\gamma),\\
&\frac{c}{2(1-c)^3}\widetilde{\Delta}^4 \mathds{1}_{[1,+\infty)}(\gamma)+\frac{1}{(1-c)^3}(c(b_1+b_2)\mathds{1}_{\{0\}}(\gamma)+\widetilde{\Delta}^2 \mathds{1}_{[0,1]}(\gamma))\Bigg)
\end{align*}
for $p/(n_1+n_2) \to c \in [0,1)$ as $n_1+n_2 \to \infty$.
\end{theorem}

%%%%%% Proof Theorem 3   %%%%%%%%%%%%%%%%%
\begin{proof}
The application of Theorem \ref{th1_3} leads to
\begin{eqnarray*}
&&p^{\min(\gamma,1)/2}\left(\frac{\hat d}{p^\gamma}- \frac{n_1+n_2-2}{n_1+n_2-p-1}\frac{(-1)^{i-1}}{2}p^{-\gamma}\Delta^2\right)\\
 &{\buildrel d \over =} & p^{\min(\gamma,1)/2-1/2} \frac{\sqrt{p}}{\sqrt{n_1+n_2-p-1}} \frac{n_1+n_2-2}{\xi} \sqrt{n_1+n_2-p-1}
 \left(1- \frac{\xi}{n_1+n_2-p-1}\right)\\
 &\times&\frac{(-1)^{i-1}}{2}p^{-\gamma}\Delta^2
+ \frac{n_1+n_2-2}{\xi}\Bigg((-1)^{i-1}\frac{\lambda n_i-2}{2\lambda n_i}\\
&\times& \left(p^{\min(\gamma,1)/2-\gamma}\lambda \xi_2+2p^{\min(\gamma,1)/2-\gamma/2}\sqrt{p^{-\gamma}\Delta^2}\sqrt{\lambda}w_0+p^{\min(\gamma,1)/2-\gamma}\lambda w_0^2\right)\\
&+&\frac{(-1)^{i-1}}{\lambda n_i}p^{\min(\gamma,1)/2-\gamma/2}\sqrt{p^{-\gamma}\Delta^2}\sqrt{\lambda}w_0\Bigg)\\
&+& \frac{n_1+n_2-2}{\xi} \Bigg(\sqrt{\left(1+\frac{1}{n_{1}+n_{2}}+\frac{p-1}{n_{1}+n_{2}-p}u\right)}\\
&\times&
\sqrt{p^{\min(\gamma,1)-2\gamma}\lambda \xi_2+(p^{\min(\gamma,1)/2-\gamma/2}\sqrt{p^{-\gamma}\Delta^2}+p^{\min(\gamma,1)/2-\gamma}\sqrt{\lambda}w_0)^2} z_0 \Bigg)\\
\end{eqnarray*}

\begin{eqnarray*}&\stackrel{\mathcal{D}}{\longrightarrow}& \mathcal{N}\Bigg((-1)^{i-1}\frac{c}{1-c}\frac{b_i-2}{2b_i}(b_1+b_2)\mathds{1}_{\{0\}}(\gamma),\\
&&\frac{c}{2(1-c)^3}\widetilde{\Delta}^4 \mathds{1}_{[1,+\infty)}(\gamma)+\frac{1}{(1-c)^3}(c(b_1+b_2)\mathds{1}_{\{0\}}(\gamma)+\widetilde{\Delta}^2 \mathds{1}_{[0,1]}(\gamma)) \Bigg),
\end{eqnarray*}
where the last line follows from Lemma 3 in \cite{bodnar2016exact} and Slutsky Theorem (see, \cite[Theorem 1.5]{DASG2008}).
%where $u|\xi_1,\xi_2,w_0 \sim \mathcal F \left(p-1, n_{1}+n_{2}-p, (n_1+n_2)^{-1}\xi_1 \right)$ with $\xi_1|\xi_2,w_0\sim \chi_{p-1,\delta^2_{\xi_2,w_0}}$ and
%$\delta^2_{\xi_2,w_0}=\frac{n_1n_2}{n_i^2}\frac{\Delta\xi_2}{\lambda \xi_2+(\sqrt{\Delta}+\sqrt{\lambda}w_0)^2}$, $z_0,w_0 \sim \mathcal N (0, 1)$, $\xi \sim \chi^2_{n_1+n_2-p-1}$, $\xi_2=\chi^2_{p-1}$; $\xi$, $z_0$ are independent of $u,\xi_1,\xi_2,w_0$ where $\xi_2$ and $w_0$ are independent as well.
\end{proof}

The parameters of the limit distribution derived in Theorem \ref{th3} can be significantly simplified in the special case of $n_1=n_2$ because of $\lambda n_1=\lambda n_2=2$. The results of Theorem \ref{th3} are also used to derived the approximate error rate for the decision \eqref{eq_est}. Let $a=\frac{1}{1-c}\frac{1}{2}p^{-\gamma}\Delta$. Then, the error rate is given by
\begin{eqnarray*}
ER_s(\Delta)&=&\frac{1}{2}\mathbb{P}\left\{\hat d >0|i=2\right\}+\frac{1}{2}\mathbb{P}\left\{\hat d \le 0|i=1\right\}\\
&=& \frac{1}{2}\mathbb{P}\left\{p^{\min(\gamma,1)/2}\left(\frac{\hat d}{p^\gamma}-(-1)^{i-1}a\right) >-p^{\min(\gamma,1)/2}(-1)^{i-1}a|i=2\right\}\\
&+&\frac{1}{2}\mathbb{P}\left\{p^{\min(\gamma,1)/2}\left(\frac{\hat d}{p^\gamma}-(-1)^{i-1}a\right) \le -p^{\min(\gamma,1)/2}(-1)^{i-1}a|i=1\right\}\\
&\approx& \frac{1}{2}\left(1-\Phi\left(\frac{ap^{\min(\gamma,1)/2}-m_2}{v}\right)\right)+\frac{1}{2}\Phi\left(\frac{-ap^{\min(\gamma,1)/2}-m_1}{v}\right),
\end{eqnarray*}
with
\begin{align*}
m_1&= \frac{c}{1-c} \frac{b_1-2}{2b_1}(b_1+b_2)\mathds{1}_{\{0\}}(\gamma),~~ m_2= -\frac{c}{1-c} \frac{b_2-2}{2b_2}(b_1+b_2)\mathds{1}_{\{0\}}(\gamma),\\
v^2&= \frac{c}{2(1-c)^3}(p^{-\gamma}\Delta^2)^2 \mathds{1}_{[1,+\infty)}(\gamma)+\frac{1}{(1-c)^3}(c(b_1+b_2)\mathds{1}_{\{0\}}(\gamma)+p^{-\gamma}\Delta^2
 \mathds{1}_{[0,1]}(\gamma)),
\end{align*}
where we approximate $\widetilde{\Delta}^2$ by $p^{-\gamma}\Delta^2$.

In the special case of $n_1=n_2$ which leads to $b_1=b_2=2$, we get
\begin{eqnarray*}
ER_s(\Delta)&=& \Phi\left(-h_c\frac{\Delta}{2}\right)
\end{eqnarray*}
with
\begin{eqnarray*}
h_c=\frac{p^{\min(\gamma,1)/2-\gamma}\sqrt{1-c}\sqrt{p^{-\gamma}\Delta^2}}{\sqrt{c(p^{-\gamma}\Delta^2)^2 \mathds{1}_{[1,+\infty)}(\gamma)/2+4c\mathds{1}_{\{0\}}(\gamma)+p^{-\gamma}\Delta^2
 \mathds{1}_{[0,1]}(\gamma)}},
\end{eqnarray*}
which is always smaller than one. Furthermore, for $\gamma \in (0,1)$ we get $h_c=\sqrt{1-c}$.

In Figure \ref{fig2}, we plot $ER_s(\Delta)$ as a function of $\Delta \in[0,100]$ for $c \in \{0.1, 0.5,0.8,0.95\}$. We also add the plot of $ER_p(\Delta)$ in order to compare the error rate of the two decision rules. Since only finite values of $\Delta$ are considered in the figure we put $\gamma=0$ and also choose $n_1=n_2$. Finally, the ratio $\frac{n_1+n_2-2}{n_1+n_2-p-1}$ in the definition of $a$ is approximated by $\frac{1}{1-c}$. We observe that $ER_s(\Delta)$ lies very close to $ER_p(\Delta)$ for $c=0.1$. However, the difference between two curves becomes considerable as $c$ growths, especially for $c=0.95$ and larger values of $\Delta$.

\begin{figure}
 \centerline{
\includegraphics[width=12cm]{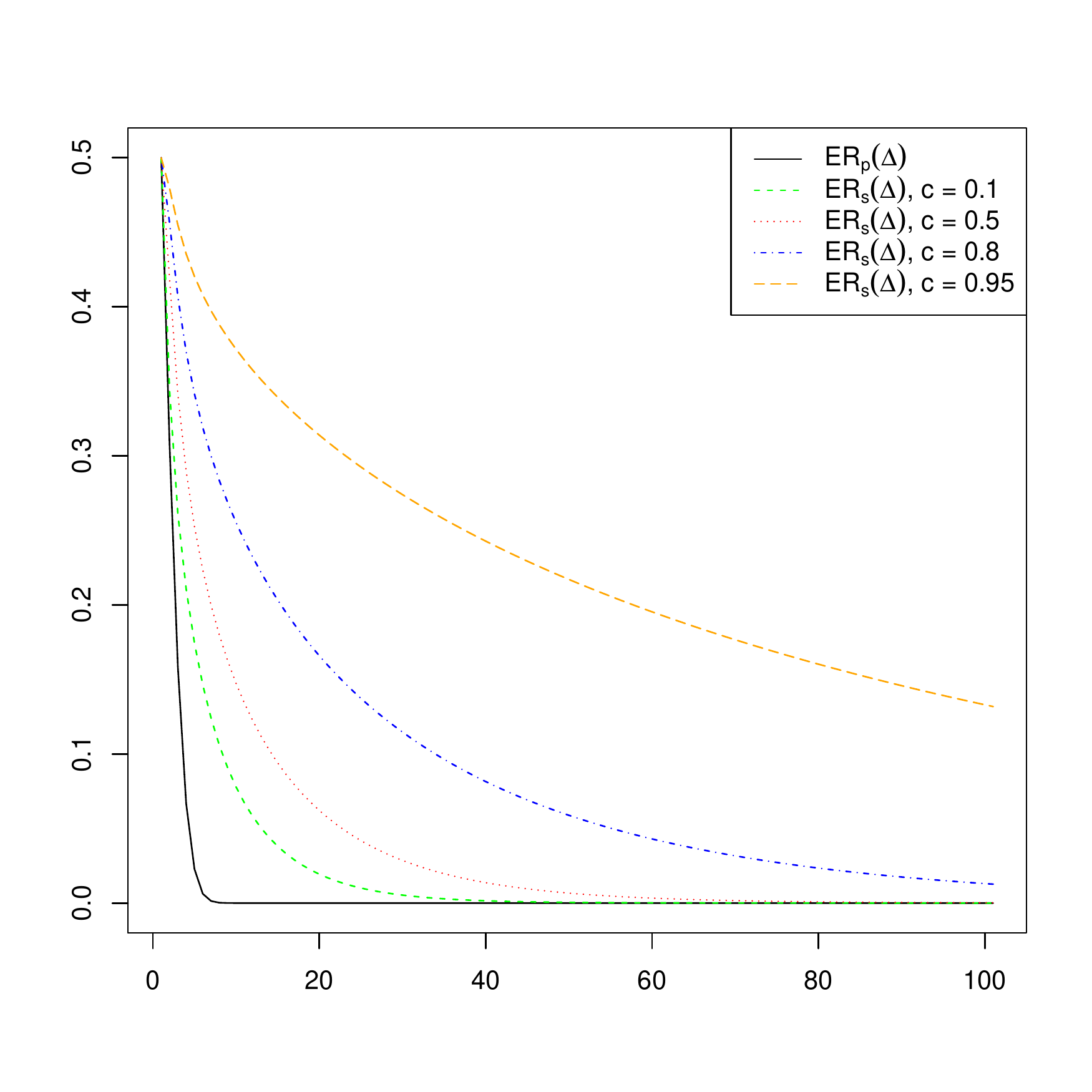}
}
 \caption{Error rates $ER_p(\Delta)$ and $ER_s(\Delta)$ as functions of $\Delta$ for $c \in \{0.1, 0.5,0.8,0.95\}$.}
\label{fig2}
 \end{figure}

\subsection{Finite-sample performance} \labmarg{s5}
In this section we present the results of the simulation study. The aim is to investigate how good the asymptotic distribution of a linear combination of the discriminant function coefficients $\hat \theta=\mathbf l^T \hat{\mathbf a}$ performs in the case of the finite dimension and of the finite sample size. For that reason we compare the asymptotic distribution of the standardized $\hat \theta$ as given in Theorem \ref{th2} to the corresponding exact distribution obtained as a kernel density approximation with the Eppanechnikov kernel applied to the simulated data from the standardized exact distribution which are generated following the stochastic representation of Corollary \ref{c1}: (i) first, $\xi_b,z_{0;b},u_b$ are sampled independently from the corresponding univariate distributions provided in Corollary \ref{c1}; (ii) second, $\hat \theta_b$ is computed by using \eqref{stoch_pres_theta} and standardized after that as in Theorem \ref{th2}; (iii) finally, the previous two steps are repeated for $b=1,...,B$ times to obtain a sample of size $B$. It is noted that $B$ could be large to ensure a good performance of the kernel density estimator.

In the simulation study, we take $\mathbf{l}=\mathbf{1}_{p}$ ($p$-dimensional vector of ones). The elements of $\boldsymbol{\mu}_{1}$ and $\boldsymbol{\mu}_{2}$ are drawn from the uniform distribution on $\left[-1, 1\right]$ when $\gamma >0$, while the first ten elements of $\bmu_1$ and the last ten elements of $\bmu_2$ are generated from the uniform distribution on $\left[-1, 1\right]$ and the rest of the components are taken to be zero when $\gamma=0$. We also take $\mathbf{\Sigma}$ as a diagonal matrix, where every element is uniformly distributed on $\left(0, 1\right]$. The results are compared for several values of $c =\{0.1, 0.5, 0.8, 0.95\}$ and the corresponding values of $p, n_1,n_2$. Simulated data consist of $N=10^{5}$ independent repetitions. In both cases $\gamma=0$ and $\gamma>0$ we plot two asymptotic density functions to investigate how robust are the obtained results to the choice of $\gamma$.

In Figures \ref{f1}-\ref{f2}, we present the results in the case of equal and large sample sizes (data are drawn with $\gamma=0$ in Figure \ref{f1} and with $\gamma>0$ in Figure \ref{f2}), while the plots in Figure \ref{f3} correspond to the case of one small sample and one large sample. We observe that the impact of the incorrect specification of $\gamma$ is not large, while some deviations are observed in Figure \ref{f3} for small values of $c$. If $c$ increases, then the difference between the two asymptotic distributions becomes negligible. In contrast, larger differences between the asymptotic distributions and the finite-sample one are observed for $c=0.8$ and $c=0.95$ in all figures, although their sizes are relatively small even in such extreme case.

%%%%%%%%%%%%%%%%%%%%%%%%%%%%%%%%%%%%%%%%%%%%%%%%%%%%%%%%%%%%%%%%%%%%
%\bibliographystyle{unsrtnat}
\bibliographystyle{apalike}
\bibliography{Ref}

%%%%%%%%%%%PLOTS%%%%%
%%%%%%%%% Figure 1:   PLOT \gamma=0       %%%%%%%%
\begin{figure}
 \centerline{
\includegraphics[width=8.8cm]{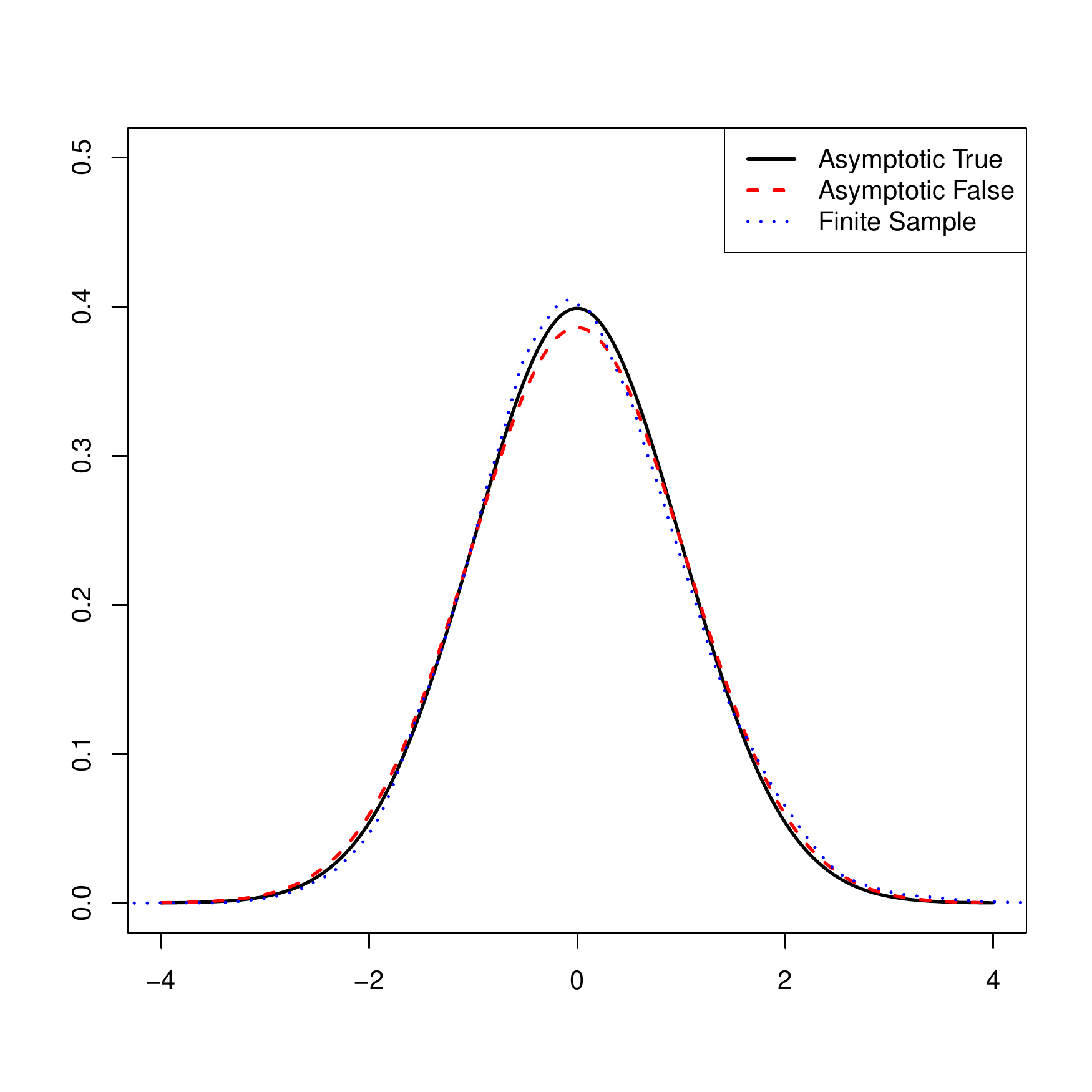}
 \includegraphics[width=8.8cm]{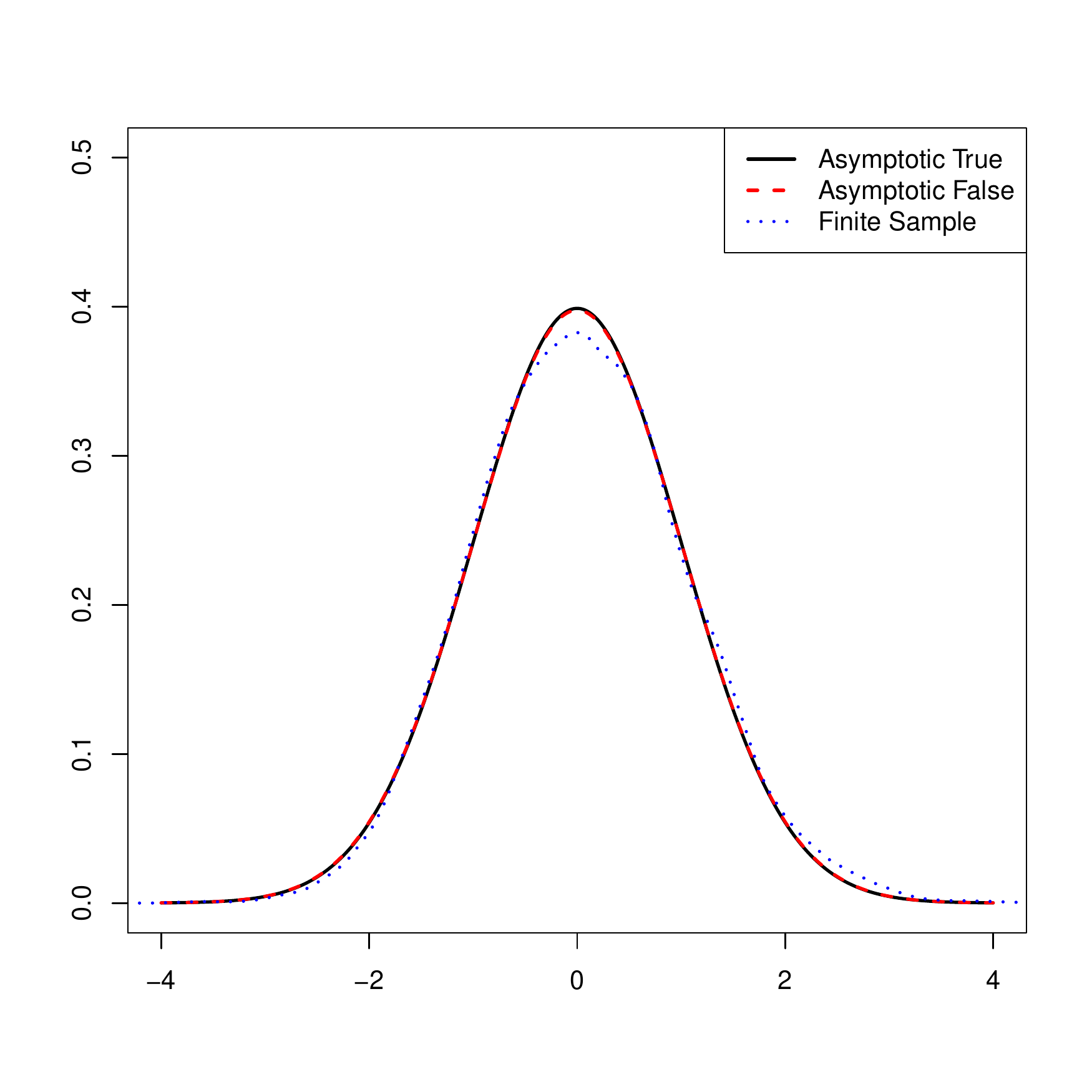}
}
${ a)\;\; p=50,~ n_{1}=n_{2}=250} \hspace{115pt}  { b)\;\; p=250,~ n_{1}=n_{2}=250}$
 \centerline{
\includegraphics[width=8.8cm]{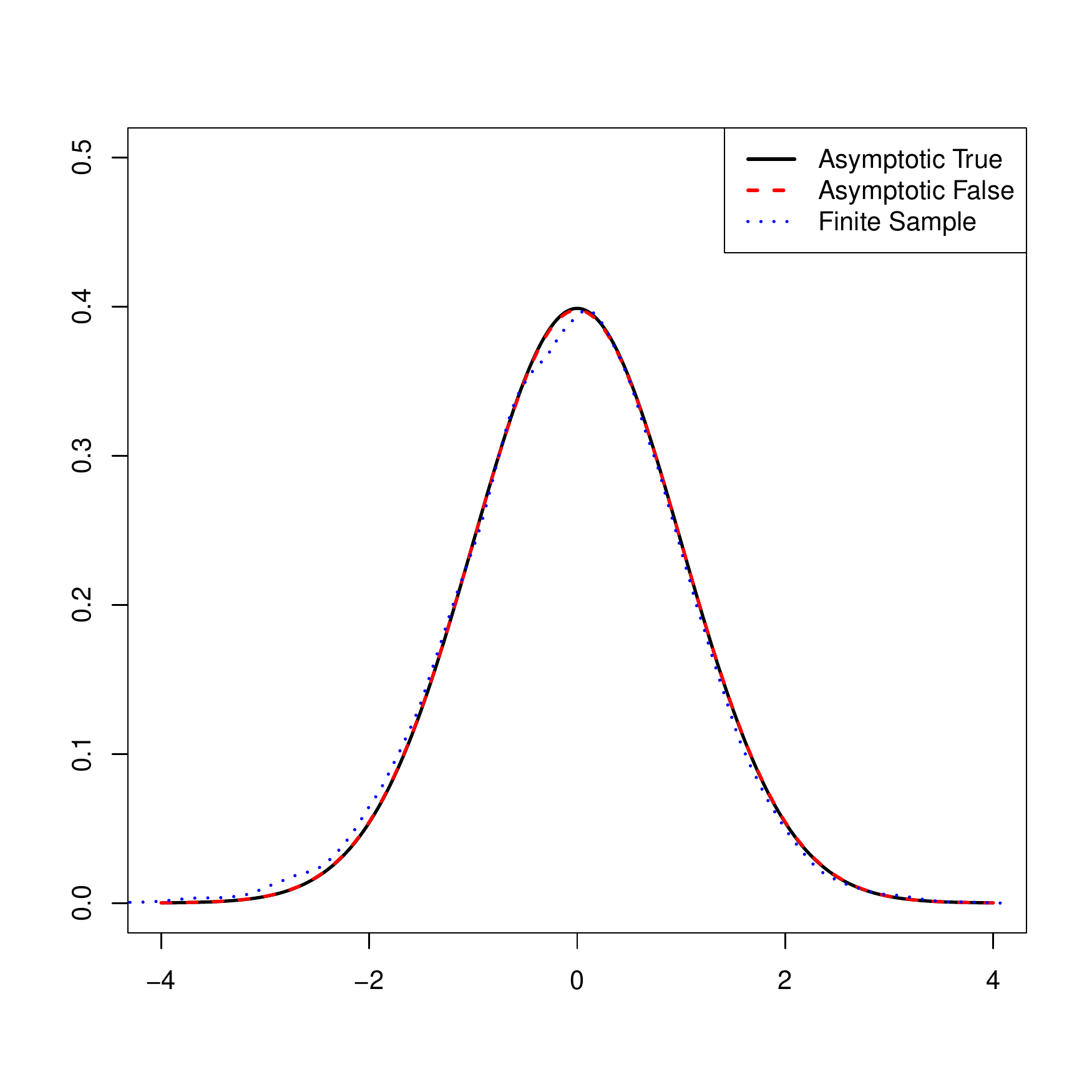}
 \includegraphics[width=8.8cm]{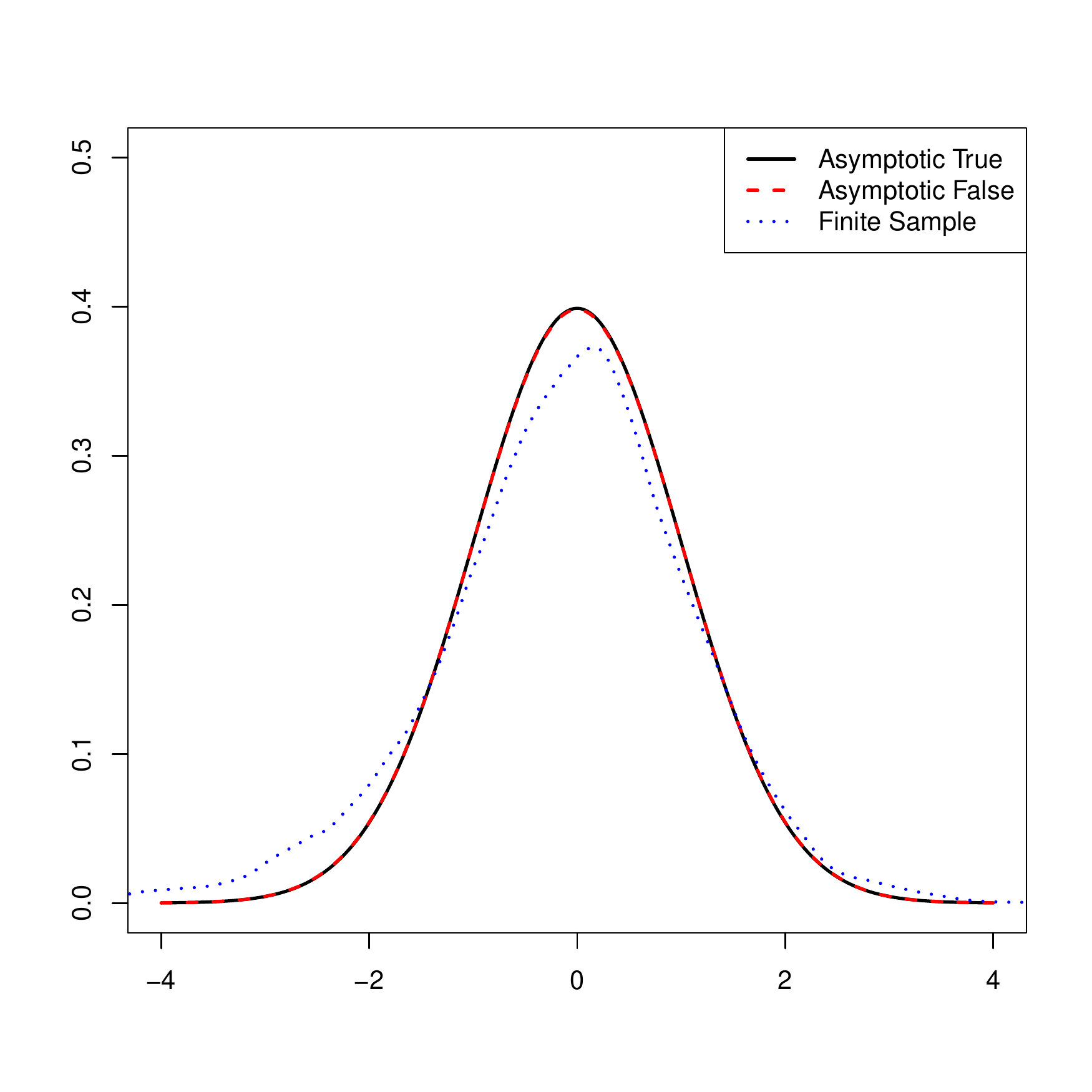}
}\\
${ c)\;\; p=400, ~n_{1}=n_{2}=250} \hspace{115pt}  { d)\;\; p=475, ~n_{1}=n_{2}=250}$
 \caption{The kernel density estimator of the asymptotic distribution and standard normal for $\hat \theta$ as given in Theorem~\ref{th2} for $\gamma=0$ and $c=\{ 0.1, 0.5 , 0.8, 0.95\}$.}
\label{f1}
 \end{figure}

%%%%%%%%% Figure 2:   PLOT \gamma>0       %%%%%%%%
\begin{figure}
 \centerline{
\includegraphics[width=8.8cm]{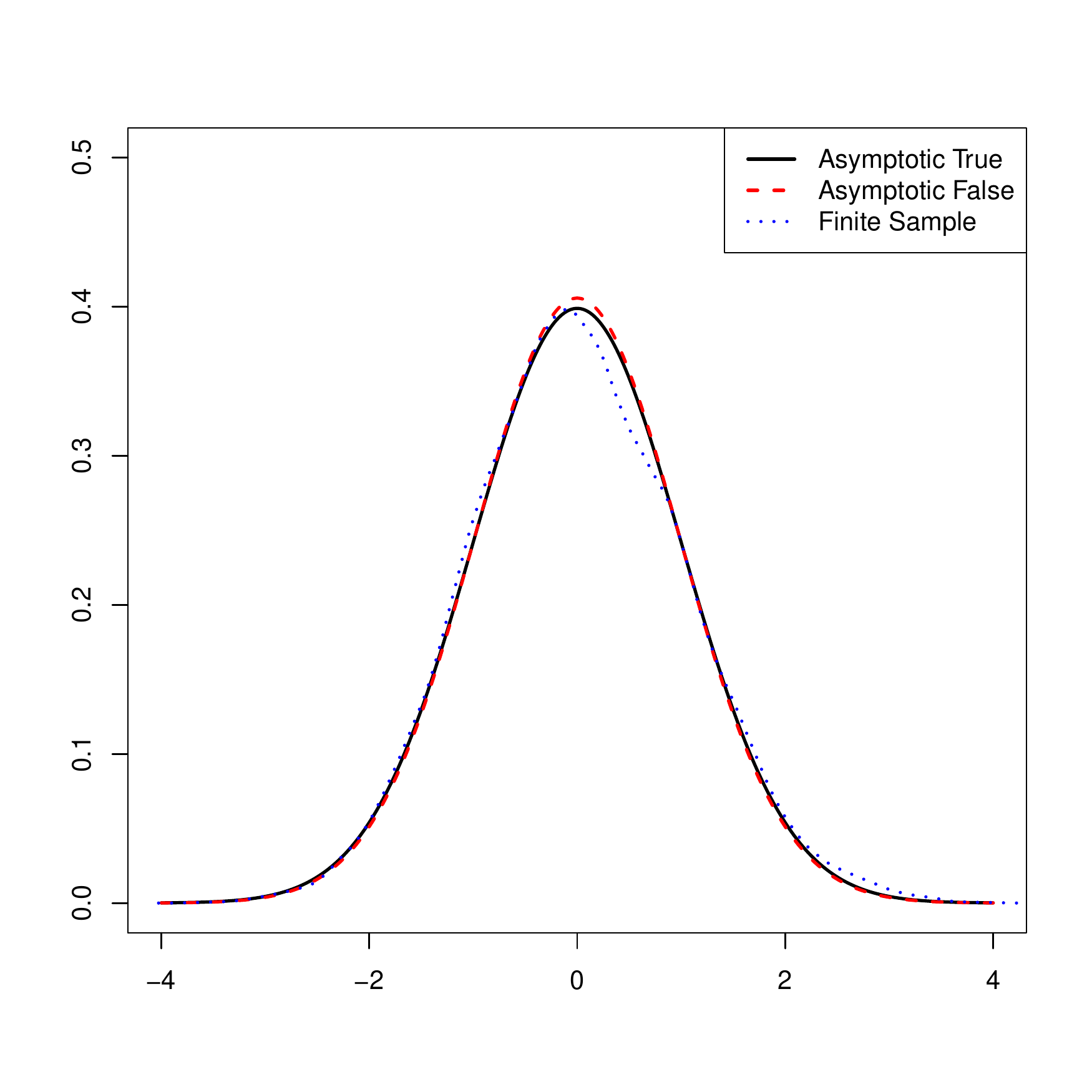}
 \includegraphics[width=8.8cm]{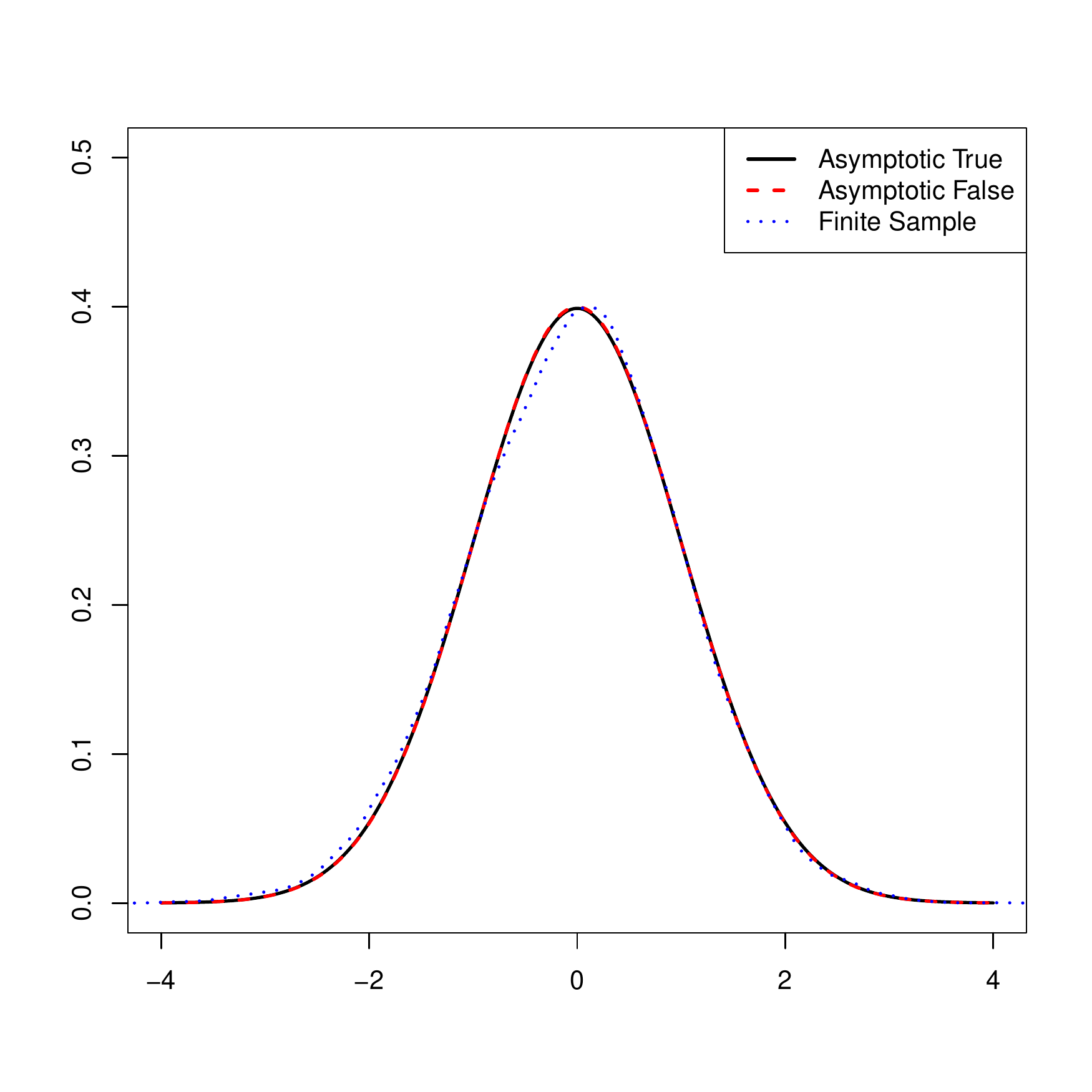}
}
${ a)\;\; p=50,~ n_{1}=n_{2}=250} \hspace{115pt}  { b)\;\; p=250,~ n_{1}=n_{2}=250}$
 \centerline{
\includegraphics[width=8.8cm]{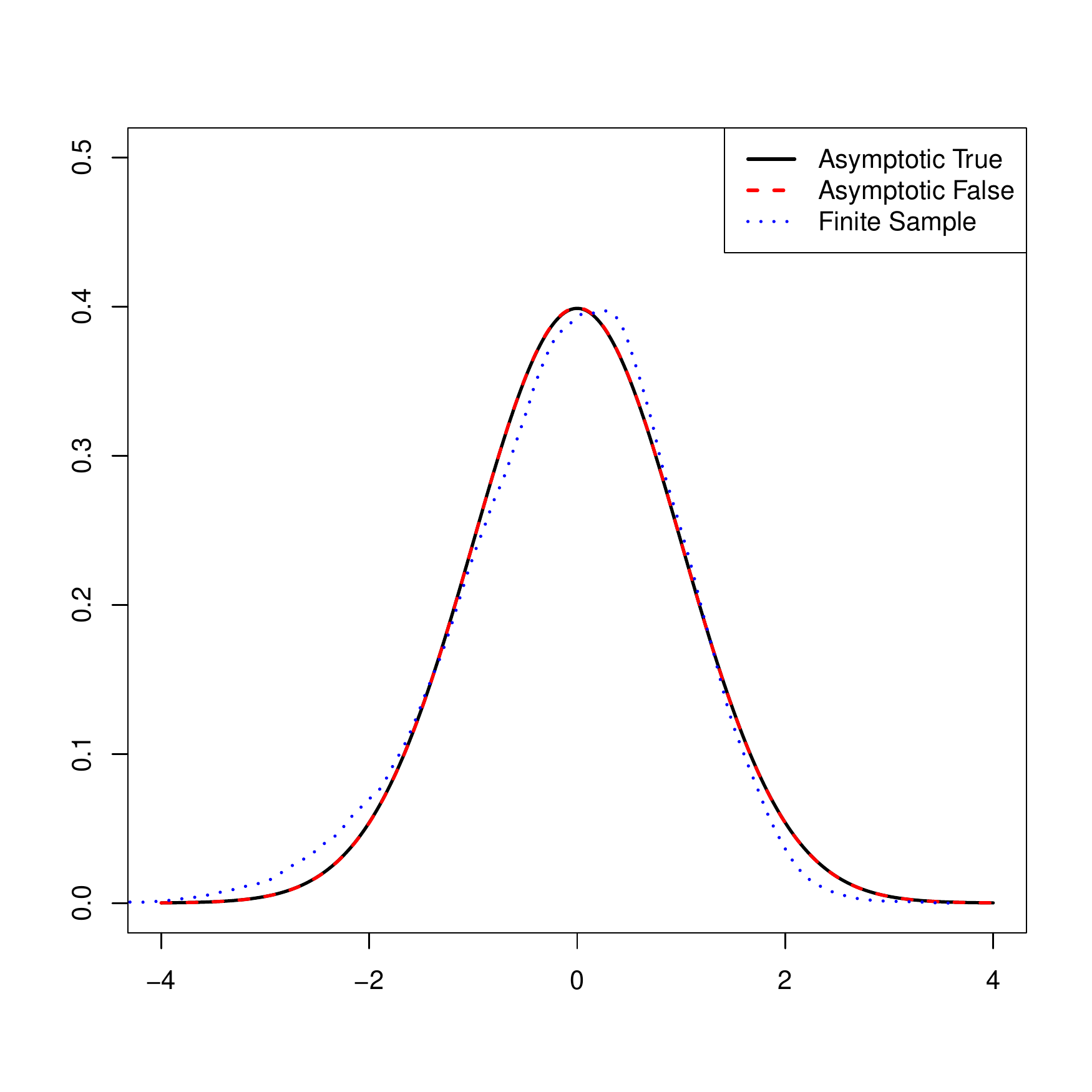}
 \includegraphics[width=8.8cm]{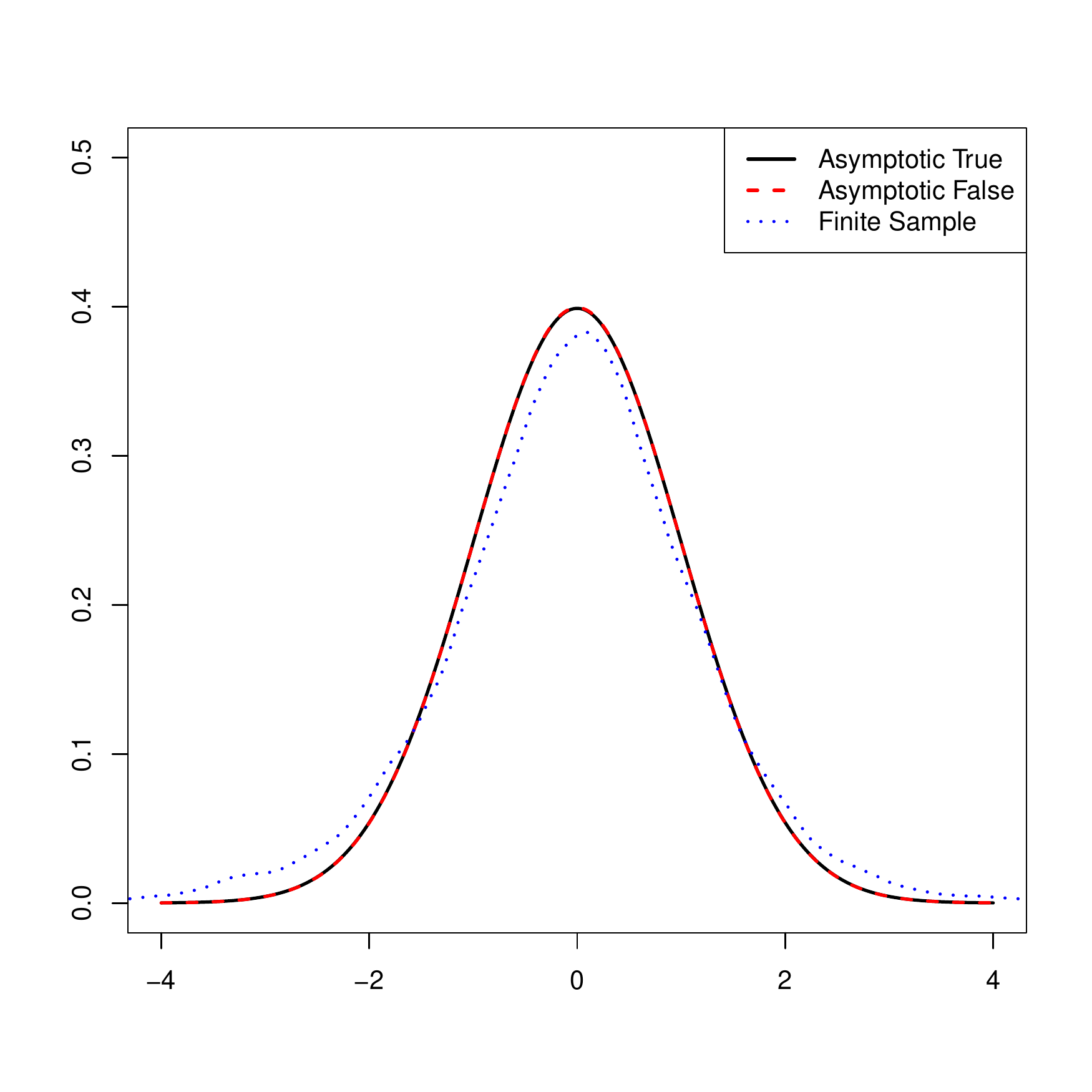}
}\\
${ c)\;\; p=400,~ n_{1}=n_{2}=250} \hspace{115pt}  { d)\;\; p=475, ~n_{1}=n_{2}=250}$
 \caption{The kernel density estimator of the asymptotic distribution and standard normal for $\hat \theta$ as given in Theorem~\ref{th2} for $\gamma > 0$ and $c=\{ 0.1, 0.5 , 0.8, 0.95\}$.}
\label{f2}
 \end{figure}

%%%%%%%%% Figure 3:   PLOT \gamma>0       %%%%%%%%
\begin{figure}
 \centerline{
\includegraphics[width=8.8cm]{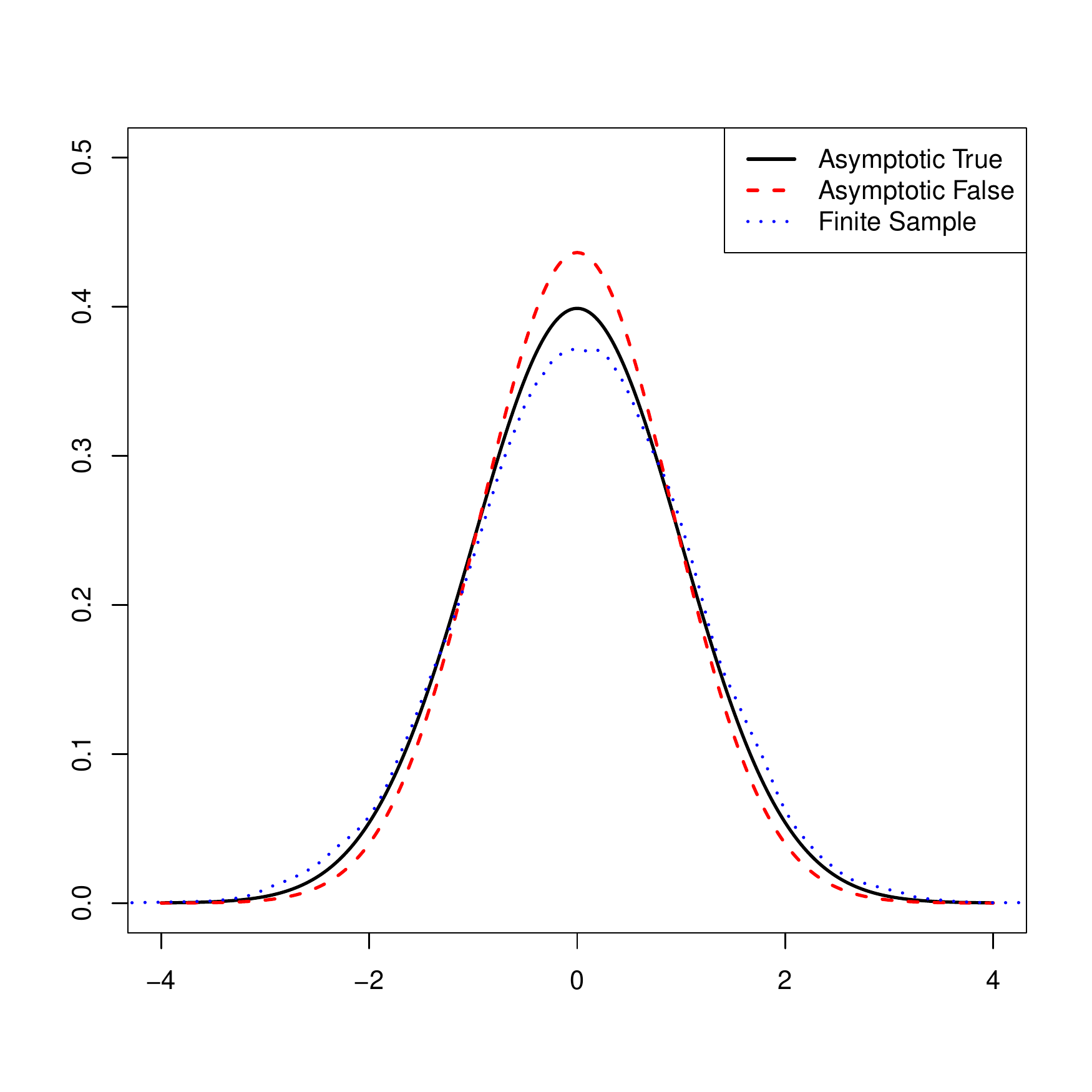}
 \includegraphics[width=8.8cm]{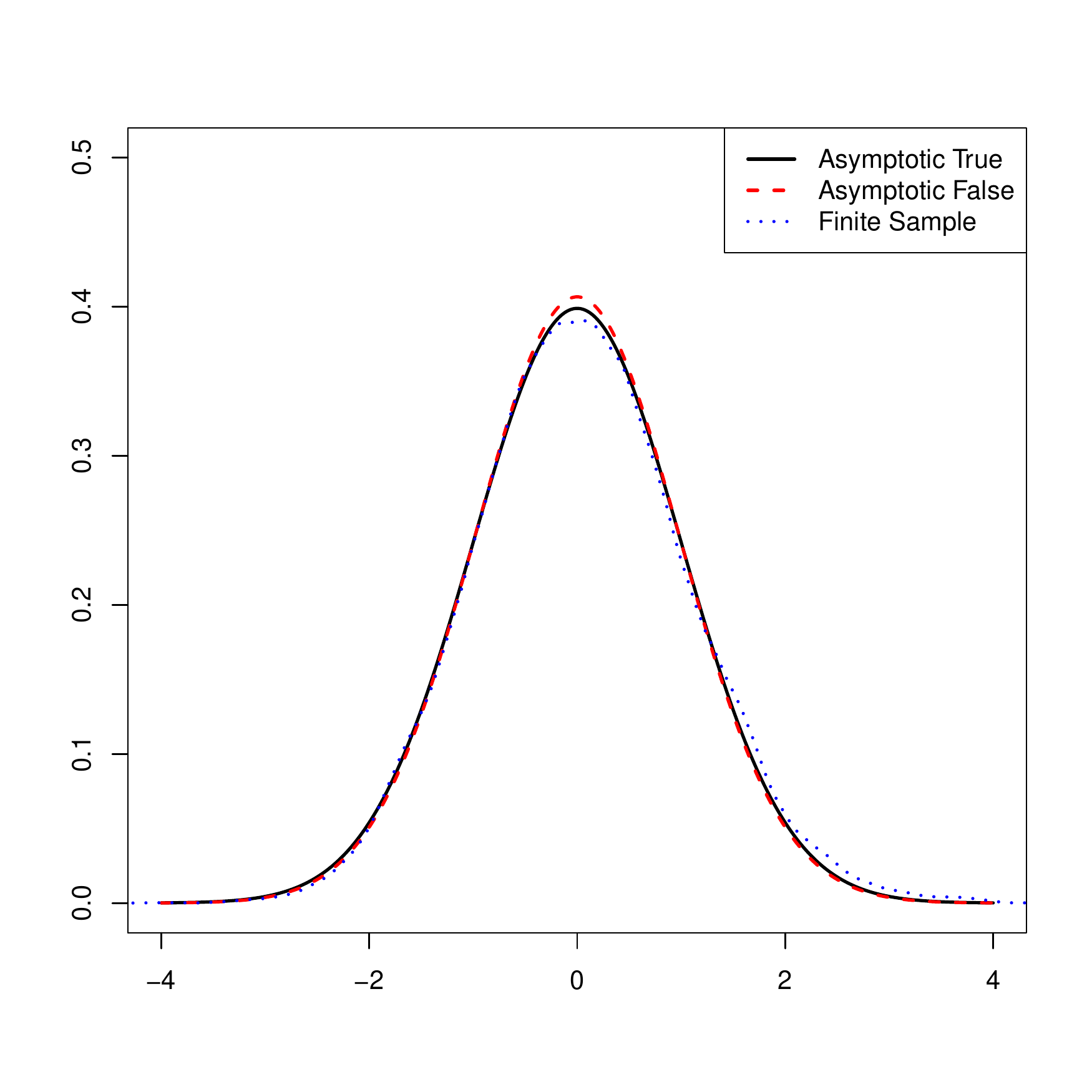}
}
${ a)\;\; p=50,~n_{1}=25,~n_{2}=475} \hspace{115pt}  { b)\;\; p=250, ~n_{1}=25,~n_{2}=475}$
 \centerline{
\includegraphics[width=8.8cm]{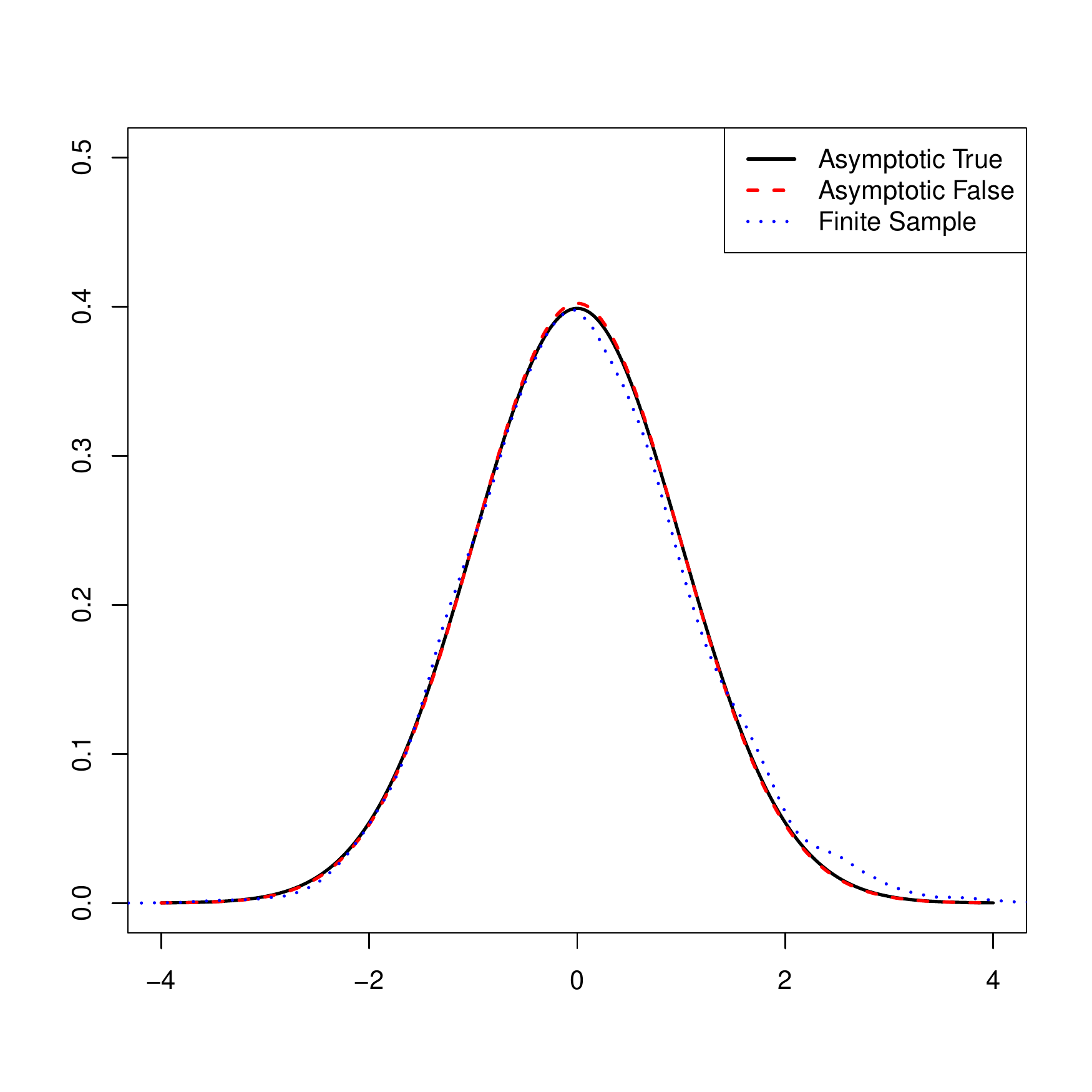}
 \includegraphics[width=8.8cm]{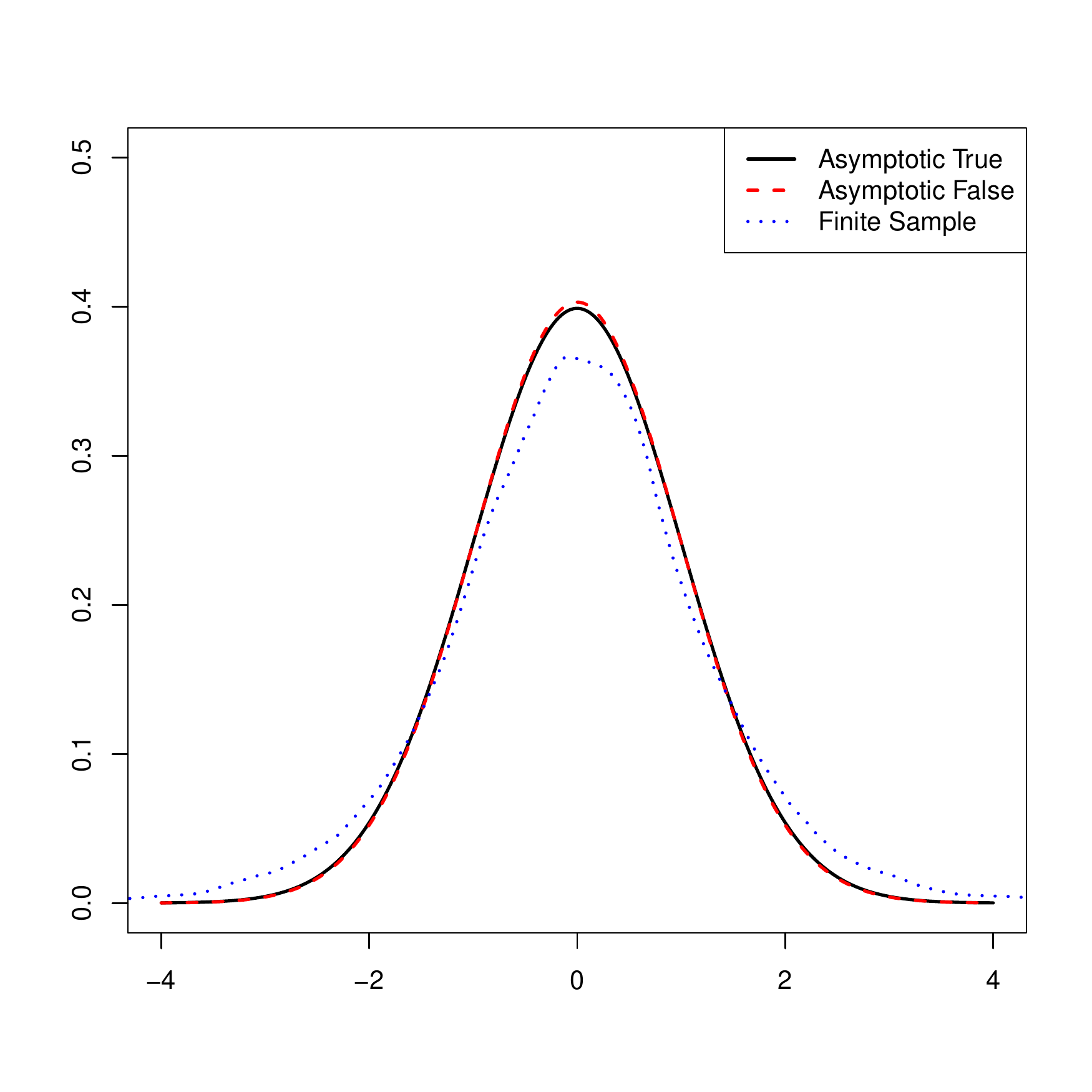}
}\\
${ c)\;\; p=400,~n_{1}=25,~n_{2}=475} \hspace{115pt}  { d)\;\; p=475,~n_{1}=25,~n_{2}=475}$
 \caption{The kernel density estimator of the asymptotic distribution and standard normal for $\hat \theta$ as given in Theorem~\ref{th2} for $\gamma > 0$ and $c=\{ 0.1, 0.5 , 0.8, 0.95\}$.}
\label{f3}
 \end{figure}

%%%%%%HISTOGRAMS%%%%%

%%%%%%%%%%%PLOTS%%%%%

\end{document}